\documentclass[psamsfonts,12pt]{amsart}
\usepackage[T1]{fontenc}
\usepackage{charter}
\usepackage[margin=1in]{geometry}  
\usepackage{graphicx}              
\usepackage{amsmath}            
\usepackage{amsfonts}              
\usepackage{amsthm}                
\usepackage{verbatim}              
\usepackage{indentfirst}           
\usepackage{underscore}
\usepackage{enumitem}
\usepackage{mathtools} 
\usepackage{amssymb}
\usepackage[all]{xy}
\usepackage{hyperref}
\usepackage{caption}
\usepackage{subcaption}
\usepackage{setspace}
\usepackage{mathrsfs}
\usepackage{pinlabel}
\usepackage[percent]{overpic}
\usepackage{float}

\DeclareMathAlphabet{\mathpzc}{OT1}{pzc}{m}{it}

\usepackage{tikz}
\usepackage{tkz-euclide} 
\usetkzobj{all} 
\tikzset{hidden/.style = {thick, dashed}}
\usetikzlibrary{plotmarks}
\usetikzlibrary{arrows}

\newtheorem{thm}{Theorem}[section]
\newtheorem{lem}[thm]{Lemma}
\newtheorem{prop}[thm]{Proposition}
\newtheorem*{definition*}{Definition}
\newtheorem{cor}[thm]{Corollary}
\newtheorem{conj}[thm]{Conjecture}

\theoremstyle{remark}
       \newtheorem*{rmk}{Remark}
\theoremstyle{remark}

\newcommand{\C}{\mathbb{C}}
\newcommand{\fr}{\frac}

\newcommand{\CC}{\mathbb{C}}

\newcommand{\QQ}{\mathbb{Q}}

\usepackage{amssymb,fge}

\title{Dynamical uniform boundedness and the $abc$-conjecture}
\author{Nicole R. Looper}

\begin{document} 

\thanks{\indent The author's research was supported by NSF grant DMS-1803021.}

\maketitle

\begin{abstract} We address the Uniform Boundedness Conjecture of Morton and Silverman in the case of unicritical polynomials, assuming a generalization of the $abc$-conjecture. For unicritical polynomials of degree at least five, we require only the standard $abc$-conjecture.
\end{abstract}

\section{Introduction}

This article studies the Uniform Boundedness Conjecture of Morton and Silverman \cite{MortonSilverman}, one of the major unsolved problems in arithmetic dynamics. The conjecture concerns preperiodic points of rational maps defined over number fields, and constitutes a massive generalization of the famous theorems of Mazur \cite{Mazur} and Merel \cite{Merel} on torsion points of elliptic curves. 

\begin{conj}[\cite{MortonSilverman},Uniform Boundedness Conjecture]{\label{conj:MortonSilverman}} Let $D,N\ge 1$, $d\ge2$ be integers. There is a constant $B=B(D,N,d)$ such that for any number field $K$ with $[K:\QQ]=D$, and any morphism $\phi:\mathbb{P}^N\to\mathbb{P}^N$ of degree $d$ defined over $K$, $\phi$ has at most $B$ preperiodic points in $\mathbb{P}^N(K)$.
\end{conj} 

Merel's theorem is an immediate consequence of a very particular case, namely, $N=1$, $d=4$, and $\phi$ in the Latt\`{e}s family arising from the multiplication-by-two maps $[2]:E\to E$ of elliptic curves $E/K$. 

Poonen has gathered evidence of Conjecture \ref{conj:MortonSilverman} in the case $K=\QQ$, $N=1$, $d=2$, for the maps $f_c(z)=z^2+c\in\QQ[z]$. If $f_c$ never has $\QQ$-rational periodic points of period at least 6, then it never has more than nine $\QQ$-rational preperiodic points \cite{Poonen}. In \cite{DoylePoonen}, Doyle and Poonen further show that for the maps $f_{d,c}=z^d+c$ defined over a number field $K$, uniform boundedness for preperiodic points reduces to the uniform boundedness of periodic points. They also prove an analogous version of the Uniform Boundedness Conjecture for non-isotrivial maps $z^d+c$ defined over a one-variable function field. Moreover, for a given $d\ge2$, Benedetto proves in \cite{Benedetto} that for any global field $K$, the number of $K$-rational preperiodic points of a degree $d$ polynomial $f\in K[z]$ is at most $O(s\log s)$, where $s$ is the number places of bad reduction of $f$. 

Our goal is to prove the following theorem.

\begin{thm}{\label{thm:UBCunicrit}} Let $K$ be a number field or a one-variable function field of characteristic zero, and let $d\ge2$. Let $f(z)=z^d+c\in K[z]$, where $f$ is not isotrivial if $K$ is a function field. If $d\ge5$ and $K$ is a number field, assume the $abc$-conjecture for $K$. If $2\le d\le4$, assume the $abcd$-conjecture (Conjecture \ref{conj:nconj}). There is a $B=B(d,K)$ such that $f$ has at most $B$ preperiodic points contained in $K$.\end{thm}

The primary arithmetic ingredient in our main theorem is Conjecture \ref{conj:nconj}, which is a generalization of the $abc$-conjecture. We will dub it ``the $abcd$-conjecture.'' It can be shown that Conjecture \ref{conj:nconj} is a special case of Vojta's conjecture with truncated counting function \cite{Vojta}.

\begin{rmk} It is easy to show that Theorem \ref{thm:UBCunicrit} applies equally to affine conjugates of $f(z)=z^d+c$, as any $g(z)\in K[z]$ that is $\overline{K}$-conjugate to $f$ is in fact $K$-conjugate to $f$. We also remark that when $K$ is a function field, Theorem \ref{thm:UBCunicrit} is superseded by \cite[Theorem 1.6]{DoylePoonen}.\end{rmk}

In the quadratic case, the proof proceeds by proving, for a unicritical map $f$ satisfying certain conditions, the existence of an arithmetically special `hexagon' whose vertices are $K$-rational preperiodic points.  This hexagon is constructed in such a way as to contradict the $abcd$-conjecture (Conjecture \ref{conj:nconj}). Other key ingredients in the proof include equidistribution results for the equilibrium measures of unicritical maps. In particular, in Sections \ref{section:potentialthy} and \ref{section:adelic}, we transpose local equidistribution results on $v$-adic Julia sets into global information holding uniformly across a large proportion of the places of bad reduction. When $f$ has a very large number of $K$-rational preperiodic points, typical elements of the form $P_i-P_j$ for preperiodic points $P_i,P_j\in K$ have relatively little prime support at the places of good reduction for $f$. Furthermore, the hexagons considered have side lengths that behave in a suitably random manner across the places of bad reduction. These two global phenomena, each originating from local equidistribution behavior, ultimately furnish the aforementioned special hexagon. The proof in the cubic and quartic cases can be simplified, with quadrilaterals replacing hexagons. When $d\ge5$, a slightly modified version of these strategies, using triangles and the $abc$-conjecture, suffices to prove the theorem. In the interest of generating as many approaches to Conjecture \ref{conj:MortonSilverman} as possible, however, we choose to exhibit an independent, purely algebraic proof that is applicable to this case.

Conjecture \ref{conj:nconj} is used to circumvent the main obstacle in proving results in the direction of Conjecture \ref{conj:MortonSilverman}: proving the uniform boundedness of $K$-rational preperiodic points across families $\mathcal{F}$ of maps defined over $K$ \emph{without} imposing a bound on the number of places of bad reduction of the maps $f\in\mathcal{F}$, or stipulating that some place `dominates' $h(c)$ by a given amount (see \S\ref{section:easycase}). We take a novel approach to avoid this limitation, at the expense of assuming the $abc$-conjecture. Many of the arguments used for $2\le d\le4$ can be generalized to other families of polynomial maps.

It is interesting to note that Frey \cite{Frey1,Frey2} uses the $abc$-conjecture to prove uniform bounds on torsion points on elliptic curves over number fields, while Hindry and Silverman \cite{HindrySilverman} use the $abc$-conjecture to prove uniform bounds on integral points on elliptic curves of bounded rank. This observation, coupled with the methods in this paper, leads to a reasonable expectation that the $abc$-conjecture and its relatives in Diophantine geometry form an auspicious avenue for proving uniform boundedness across one-parameter families of maps.

\indent\textbf{Acknowledgements.} I would like to thank Rob Benedetto, Laura DeMarco, Holly Krieger, Joe Silverman, and Tom Tucker for useful discussions relating to this project. I thank Holly Krieger in particular for extensive and fruitful conversations regarding the arguments presented in \S\ref{section:potentialthy}, and Joe Silverman for his many helpful comments on a draft of this article.

\section{Background}

\subsection{Notation} We set the following notation:

\setlength{\tabcolsep}{15pt}

\begin{tabular}{r p{10cm}}
	
	$K$ & a number field or a one-variable function field, i.e., a finite extension of a field $k(t)$ of rational functions in one variable over $k$\\ $M_K$ & a complete set of inequivalent places of $K$, with absolute values $|\cdot|_v$ normalized to extend the standard absolute values on $\QQ$ if $K$ is a number field, or $k(t)$ if $K$ is a function field\\ $M_K^0$ & the set of nonarchimedean places of $K$  \\ $M_K^\infty$ & the set of archimedean places of $K$ \\ 
	$k_{\mathfrak{p}}$ & the residue field associated to the finite prime $\mathfrak{p}$ of $K$ \\
	$N_\mathfrak{p}$ & $\fr{\log(\#k_\mathfrak{p})}{[K:\QQ]}$ if $K$ is a number field, \\ & $\fr{[k_\mathfrak{p}:k]}{[K:k(t)]}$ if $K/k(t)$ is a function field \\ $r_v$ & $\fr{[K_v:\QQ_v]}{[K:\QQ]}$ if $K$ is a number field, \\ & $\fr{[K_v:k(t)_v]}{[K:k(t)]}$ if $K/k(t)$ is a function field \\ $\lambda_v(\cdot)$ & $\log\max\{1,|\cdot|_v\}$ \\ $v_\mathfrak{p}$ & the standard $\mathfrak{p}$-adic valuation on $K$
\end{tabular} 
\newline

\noindent For all $v=\mathfrak{p}\in M_K^0$, one has \[r_v\lambda_v(z)=-\min\{0,v_\mathfrak{p}(z)\}N_\mathfrak{p}.\] Let $\mathcal{O}_K$ denote the ring of integers if $K$ is a number field, and the valuation ring if $K$ is a function field.
If $K$ is a number field, $n\ge 2$ and $P=(z_1,\dots,z_n)\in\mathbb{P}^{n-1}(K)$ with $z_1,\dots,z_n\in K$, let \begin{equation*}\begin{split}h(P)=&\sum_{\textup{primes }\mathfrak{p} \textup{ of } \mathcal{O}_K} -\min\{v_{\mathfrak{p}}(z_1),\dots,v_{\mathfrak{p}}(z_n)\}N_{\mathfrak{p}}\\&+\dfrac{1}{[K:\QQ]} \sum_{\sigma:K\hookrightarrow\CC} \log\max\{|\sigma(z_1)|,\dots,|\sigma(z_n)|\},\end{split}\end{equation*} where we do not identify conjugate embeddings. (We choose to express the height in this form, which separates the nonarchimedean and archimedean contributions, for convenience in applying the $abcd$-conjecture.) If $K$ is a function field, let \[h(P)=\sum_{\textup{primes }\mathfrak{p} \textup{ of } \mathcal{O}_K} -\min\{v_{\mathfrak{p}}(z_1),\dots,v_{\mathfrak{p}}(z_n)\}N_{\mathfrak{p}}.\] For any $P=(z_1,\dots,z_n)\in\mathbb{P}^{n-1}(K)$ with $z_1,\dots,z_n\in K^*$, we define \[I(P)=\{\textup{primes }\mathfrak{p} \textup{ of } \mathcal{O}_K\mid v_{\mathfrak{p}}(z_i)\ne v_{\mathfrak{p}}(z_j)\textup{ for some } 1\le i,j\le n\}\] and let \[\textup{rad}(P)=\sum_{\mathfrak{p}\in I(P)} N_{\mathfrak{p}}.\] 

\subsection{The $abcd$-conjecture} In order to address the cases $2\le d\le4$ in Theorem \ref{thm:UBCunicrit}, we will use a generalization of the $abc$-conjecture. The standard $abc$-conjecture corresponds to the case $n=3$. 

\begin{conj}[The $abcd$-conjecture]{\label{conj:nconj}} Let $K$ be a number field or a one-variable function field of characteristic zero, and let $n\ge 3$. Let $[Z_1:\cdots:Z_n]$ be the standard homogeneous coordinates on $\mathbb{P}^{n-1}(K)$, and let $\mathcal{H}$ be the hyperplane given by $Z_1+\dots+Z_n=0$. For any $\epsilon>0$, there is a proper Zariski closed subset $\mathcal{Z}=\mathcal{Z}(K,\epsilon,n)\subsetneq\mathcal{H}$ and a constant $C_{K,\mathcal{Z},\epsilon,n}$ such that for all $P=(z_1,\dots,z_n)\in\mathcal{H}\setminus\mathcal{Z}$ with $z_1,\dots,z_n\in K^*$, we have \[h(P)<(1+\epsilon)\textup{rad}(P)+C_{K,\mathcal{Z},\epsilon,n}.\]\end{conj} \begin{rmk} For our purposes in proving Theorem \ref{thm:UBCunicrit}, the $1+\epsilon$ appearing here cannot simply be replaced by a larger constant. Thus, for example, the results of \cite{BrownawellMasser} and \cite{Voloch} do not suffice to give an unconditional proof of Theorem \ref{thm:UBCunicrit} when $K$ is a function field and $2\le d\le4$.\end{rmk} 

We now show how one can derive Conjecture \ref{conj:nconj} from Vojta's conjecture with truncated counting function, expanding upon an argument discussed in \cite{Vojta}. Vojta's conjecture concerns the approximation of certain varieties by algebraic points in an ambient smooth projective variety $X$ defined over $K$. For a divisor $D\in\text{Div}(X)$ and $v\in M_K$, let $\lambda_{D,v}$ be a $v$-adic local height on $(X\setminus D)(K_v)$ relative to $D$. (For background on local height functions, see \cite[Chapter B.8]{HindrySilverman:DiophantineGeometry}.) For $P\in X(\overline{K})$, let $h_D(P)=\sum_{v\in M_K}r_v\lambda_{D,v}(P)$. If $L/K$ is a finite extension, and $\textbf{D}_{L/K}$ denotes the discriminant ideal of $L/K$ with respect to $\mathcal{O}_L$, let \[d_{L/K}=\fr{1}{[L:K]}\log\prod_{v\in M_K^0}|\textbf{D}_{L/K}|_v^{-r_v}\] be the logarithmic discriminant of $L/K$. We say that a reduced effective divisor $D\in\textup{Div}(X)$ is a \emph{normal crossings divisor} if $D=\sum_{i=1}^r D_i$ for irreducible subvarieties $D_i$, and the variety $\cup_{i=1}^r D_i$ has normal crossings.

\begin{definition*} Let $S\subset M_K$ be a finite set of places of $K$ containing $M_K^\infty$. For $P\in X(\overline{K})\setminus D$, and $\lambda_{D,\mathfrak{p}}$ a set of local height functions relative to $D$, the arithmetic truncated counting function is \[N_S^{(1)}(D,P)=\sum_{\substack{v=\mathfrak{p}\in M_{K(P)}\\ \mathfrak{p}\nmid S}}
	\chi(\lambda_{D,\mathfrak{p}}(P))N_\mathfrak{p}\] where for $a\in\mathbb{R}$, \[\chi(a)=\begin{cases}0 & \text{if }a\le0\\ 1 & \text{if }a>0.\end{cases}\]\end{definition*}

The form of Vojta's conjecture we consider is as follows.

\begin{conj}\cite[Conjecture 2.3]{Vojta}\label{conj:Vojta} Let $K$ be a number field or a one-variable function field of characteristic zero, and let $S$ be a finite set of places of $K$ containing the archimedean places. Let $X$ be a smooth projective variety over $K$, let $D$ be a normal crossings divisor on $X$, let $K_X$ be a canonical divisor on $X$, let $A$ be an ample divisor on $X$, let $r\in\mathbb{Z}_{>0}$, and let $\epsilon>0$. Then there exists a proper Zariski closed subset $\mathcal{Z}=\mathcal{Z}(K,S,X,D,A,r,\epsilon)\subsetneq X$ such that \begin{equation*}N_S^{(1)}(D,P)\ge h_{K_X+D}(P)-\epsilon h_A(P)-d_{K(P)/K}+O(1)\end{equation*} for all $P\in X(\overline{K})\setminus\mathcal{Z}$ with $[K(P):K]\le r$.
\end{conj}

Let $n\ge3$, let $X=\mathcal{H}$ be the hyperplane $Z_1+\dots+Z_n=0$ in $\mathbb{P}^{n-1}(\overline{K})$, and let $A$ be the hyperplane divisor $[Z_n=0]$ on $X$. The divisor $-(n-1)A$ is a canonical divisor of $X$; thus, we set $K_X=-(n-1)A$. Let $P=(z_1,\dots,z_n)\in\mathbb{P}^{n-1}(K)$. We will take \[\lambda_{A,v}(P)=\log\max_{1\le j\le n}\left|\frac{z_j}{z_n}\right|_v\] as a local height function relative to $A$. Let $D$ be the degree $n$ normal crossings divisor on $X$ given by $[Z_1Z_2\cdots Z_n=0]$. We will take as a local height function relative to $D$: \[\lambda_{D,v}(P)=\log\max_{1\le j\le n}\left|\frac{z_j^{n}}{z_1\cdots z_n}\right|_v.\] Then for $S=M_K^\infty$ and $P=(z_1,\dots,z_n)\in \mathbb{P}^{n-1}(K)$ with $z_1,\dots,z_n\in K^*$, we have \begin{equation}{\label{eqn:rad}}N_S^{(1)}(D,P)=\textup{rad}(P).\end{equation} Moreover, \begin{equation}\label{eqn:ampledivisor}h_D(P)=nh_A(P)=nh(P).\end{equation} Let $\epsilon>0$. It follows from Conjecture \ref{conj:Vojta} along with (\ref{eqn:rad}) and (\ref{eqn:ampledivisor}) that there is a proper Zariski closed $\mathcal{Z}=\mathcal{Z}(K,n,\epsilon)\subsetneq\mathcal{H}$ and a constant $C_{K,\mathcal{Z},\epsilon,n}$ such that for all $P=(z_1,\dots,z_n)\in\mathcal{H}(K)\setminus\mathcal{Z}$ with $z_1,\dots,z_n\in K^*$, we have \[(n-(n-1)-\epsilon)h(P)\le \textup{rad}(P)+C_{K,\mathcal{Z},\epsilon,n}.\]

\subsection{Nonarchimedean potential theory} For $v\in M_K$, let $\CC_v$ denote the $v$-adic completion of an algebraic closure of $K_v$. We denote open and closed disks in $\CC_v$ as follows: \[D(a,r)=\{z\in\CC_v:|z-a|_v\le r\},\] \[D(a,r)^-=\{z\in\CC_v:|z-a|_v< r\}.\] Unless otherwise specified, we impose the convention that disks have radius belonging to the value group $|\CC_v^{\times}|$. An \emph{annulus} in $\CC_v$ is a set of the form \[A=\{z\in\CC_v:0<r_1<|z-a|_v<r_2\},\] with $r_1,r_2\in|\CC_v^{\times}|$. Its \emph{modulus} is given by $\textup{mod}(A)=\log(r_2/r_1)$. For $f(z)\in K[z]$ and $z\in\C_v$, let \[\hat{\lambda}_v(z)=\lim_{n\to\infty} \fr{1}{d^n}\lambda_v(f^n(z))\] be the standard $v$-adic escape-rate function. (See \cite[\S3.4, 3.5]{Silverman3} for a proof that the limit defining $\hat{\lambda}_v(z)$ exists.) Note that $\hat{\lambda}_v(z)$ obeys the transformation rule \begin{equation*}\hat{\lambda}_v(f(z))=d\hat{\lambda}_v(z)\end{equation*} for all $z\in \CC_v$. 

Let $v\in M_K^0$. For every $g\in\mathbb{R}_{>0}$, Proposition \ref{prop:nonarchdiskpreimage} implies that the set of points $z\in\CC_v$ such that $\hat{\lambda}_v(z)\le g$ is a finite union of $r_g$ closed disks. Call $g\in\mathbb{R}_{>0}$ a \emph{splitting potential} of $f$ if $r_{g'}>r_g$ for any $g'<g$. If $g_0$ is the greatest splitting potential of $f$, the set \[A_0=\{z\in\CC_v: g_0<\hat{\lambda}_v(z)<dg_0\}\] is called the \emph{fundamental annulus} of $f$. A disk $D_1\subseteq\CC_v$ \emph{lies at distance $r$} from a disjoint disk $D_2\subseteq\CC_v$ if for any points $z_1\in D_1$, $z_2\in D_2$, we have $|z_1-z_2|_v=r$. 

In order to carry out our potential theoretic analysis in the nonarchimedean setting, we work in Berkovich space. For $v\in M_K^0$, the Berkovich affine line $\textbf{A}_v^1$ over $\CC_v$ is the collection of equivalence classes of multiplicative seminorms on $\CC_v[T]$ which extend the norm $|\cdot|_v$ on $\CC_v$. Let $[\phi]_{D(a,r)}=\sup_{z\in D(a,r)}|\phi(z)|_v$ denote the sup-norm on the disk $D(a,r)$. The Berkovich classification theorem (see \cite[Theorem 2.2]{BakerRumely}) states that each seminorm $[\cdot]_x$ corresponds to an equivalence class of sequences of nested closed disks $\{D(a_i,r_i)\}$ in $\CC_v$, by the identification \[[\phi]_x:=\lim_{i\to\infty}[\phi]_{D(a_i,r_i)}.\] We define the action of polynomial $f\in\CC_v[T]$ on a point $x\in\textbf{A}_{\CC_v}^1$ by $[\phi]_{f(x)}=[\phi\circ f]_x$ for $\phi\in\CC_v[T]$. 

For $a\in\CC_v$, we define open and closed Berkovich disks of radius $r$ \[\mathcal{B}(a,r)^-=\{x\in\textbf{A}_v^1:[T-a]_x<r\},\] \[\mathcal{B}(a,r)=\{x\in\textbf{A}_v^1:[T-a]_x\le r\},\] corresponding to the classical disks $D(a,r)^-$ and $D(a,r)$ respectively. A basis for the open sets of $\textbf{A}_v^1$ is given by sets of the form $\mathcal{B}(a,r)^-$ and $\mathcal{B}(a,r)^-\setminus\cup_{i=1}^N\mathcal{B}(a_i,r_i)$, where $a,a_i\in \CC_v$ and $r,r_i>0$. This topology is called the \emph{Berkovich topology}. We consider $\textbf{A}_v^1$ as a measure space whose Borel $\sigma$-algebra is generated by this topology. Proposition \ref{prop:nonarchmoduli} together with \cite[Lemma 9.12]{BakerRumely} imply that if $D\subseteq\CC_v$ is a disk, and $f^{-1}(D)=\cup_iD_i$, then $f^{-1}(\mathcal{B})=\cup_i\mathcal{B}_i$, where $\mathcal{B}$ and $\mathcal{B}_i$ correspond to the classical disks $D$ and $D_i$ respectively. Let $\delta_v(z,w)$ denote the Hsia kernel relative to infinity (see \cite[Section 4.1]{BakerRumely}). A Berkovich disk $\mathcal{B}(a,r)$ is said to \emph{lie at distance} $R$ from a disjoint Berkovich disk $\mathcal{B}(b,s)$ if $\delta_v(z,w)=R$ for all $z\in\mathcal{B}(a,r), w\in\mathcal{B}(b,s)$. (If the disks are to be disjoint, we must have $R=|a-b|_v>\max\{r,s\}$.) This holds if and only if $|z-w|_v=R$ for all Type I points $z\in\mathcal{B}(a,r), w\in\mathcal{B}(b,s)$. The Berkovich $v$-adic filled Julia set of $f(z)\in K[z]$ is defined as \[\mathcal{K}_v=\bigcup_{M>0}\{x\in\textbf{A}_v^1:[f^n(z)]_x\le M\text{ for all } n\ge0\}.\]

Let $v\in M_K^0$, let $E\subseteq\mathbf{A}_v^1$, and let $\nu$ be a probability measure with support contained in $E$. The potential function of $\nu$ is by definition \[p_\nu(z)=\int_E-\log\delta_v(z,w)d\nu(w),\] and the energy integral of $\nu$ is \[I(\nu)=\int_E p_\nu(z)d\nu(z).\] The integrals here are Lebesgue integrals; the function $\delta_v(z,w)$ is upper semicontinuous (\cite[Proposition 4.1(A)]{BakerRumely}), so $-\log\delta_v(z,w)$ is lower semicontinuous, and hence Borel measurable relative to the $\sigma$-algebra generated by the Berkovich topology. 
The capacity of $E$ is \[\gamma_v(E):=e^{-\inf_\nu I(\nu)}.\] If $E$ is compact and $\gamma_v(E)>0$, there is a unique probability measure $\mu_E$ on $E$ for which $I(\mu_E)=\inf_\nu I(\nu)$ \cite[Proposition 7.21]{BakerRumely}. This measure $\mu_E$ is called the \emph{equilibrium measure} for $E$. When $E$ is compact, the capacity coincides with the quantity \[\lim_{n\to\infty}\sup\left\{\prod_{i\ne j}\delta_v(z_i,z_j)^{1/(n(n-1))}:z_1,\dots,z_n\in E\right\},\] which is known as the transfinite diameter of $E$ \cite[Theorem 6.24]{BakerRumely}. For a set $T\subseteq\mathbf{A}_v^1$ of $n$ points $z_1,\dots,z_n$, let \[d_v(T):=\prod_{i\ne j}\delta_v(z_i,z_j)^{1/(n(n-1))}.\] 

\section{Equidistribution and dynamics}{\label{section:potentialthy}}

Having introduced the key arithmetic and analytic objects used in proving Theorem \ref{thm:UBCunicrit}, we turn to dynamical considerations.

We fix throughout a degree $d\ge2$ and a product formula field $K$ (so that $K$ is either a number field or a one-variable function field by \cite[Theorem 3]{ArtinWhaples}), and let $f(z)=z^d+c\in K[z]$. The $n$-th iterate of $f$ will be denoted by $f^n$. For $\alpha\in\mathbb{P}^1(\overline{K})$, the \emph{forward orbit} of $\alpha$ is the set $\{f^n(\alpha)\}_{n=0}^\infty$. A point $\alpha\in\mathbb{P}^1(\overline{K})$ is said to be \emph{preperiodic} if its forward orbit is finite, and \emph{periodic} if $f^n(\alpha)=\alpha$ for some $n\ge1$. We say $f(z)=z^d+c\in K[z]$ has \emph{bad reduction at $v\in M_K^0$}, or alternatively that $v$ is a \emph{bad place}, if $\lambda_v(c)>0$. If $v\in M_K^0$ and $\lambda_v(c)=0$, then we say that $f$ has \emph{good reduction at $v$}. We introduce a notion measuring the size of a set of bad places for a given $f(z)=z^d+c\in K[z]$.

\begin{definition*} For $0<\delta<1$, and $\Sigma\subseteq M_K^0$, a \emph{$\delta$-slice of places} $v\in\Sigma$ is a set $S$ of bad places $v\in\Sigma$ of $f$ such that \[\sum_{v\in S}r_v\lambda_v(c)\ge\delta\sum_{v\in\Sigma}r_v\lambda_v(c).\]\end{definition*}
We will make use of the fact that pre-images of disks under polynomials behave nicely in the nonarchimedean setting. Moduli of annuli also transform functorially under covering maps, just as in the archimedean setting. \begin{prop}\cite[Lemma 2.7]{Benedetto}{\label{prop:nonarchdiskpreimage}} Let $v\in M_K^0$, and let $D^-\subseteq\C_v$ be an open disk, and let $\phi\in\C_v[z]$ be a polynomial of degree $d\ge 1$. Then $\phi^{-1}(D^-)$ is a disjoint union $D_1^-\cup\cdots\cup D_m^-$ of open disks, with $1\le m\le d$. Moreover, for each $i=1,\dots,m$, there is an integer $1\le d_i\le d$ such that every point in $D^-$ has exactly $d_i$ pre-images in $D_i^-$, counting multiplicity, and that $d_1+\cdots+d_m=d$. The foregoing also holds with the open disks $D^-$, $D_i^-$ replaced by closed disks $D$, $D_i$.
\end{prop} The $D_i^-$ in Proposition \ref{prop:nonarchdiskpreimage} are referred to as the \emph{disk components} of $\phi^{-1}(D^-)$. 

\begin{prop}\cite[Corollary 2.6]{Baker}{\label{prop:nonarchmoduli}} If $v\in M_K^0$, and $\phi\in\C_v(z)$ is a degree $k$ covering map $\phi:A_1\to A_2$ of annuli in $\C_v$, then \[\textup{mod}(A_1)=\fr{1}{k}\textup{mod}(A_2).\] \end{prop} 

We will make key use, throughout this article, of a lemma describing the shape of the filled Julia set for $f$ at places of bad reduction. The lemma says that the description of $\mathcal{K}_v$ in terms of $|c|_v$ is uniform across all bad places not dividing $d$. The reader may wish to refer to Figure \ref{figure:1}, which portrays the shape of the filled Julia set for unicritical maps $f(z)\in\QQ[z]$ in both the archimedean and nonarchimedean contexts. Figure \ref{figure:a} shows an example at the place $v=\infty$, and Figure \ref{figure:b} shows a diagram of the analogous action on the Berkovich projective line. (In Figure \ref{figure:a}, the filled Julia set is contained within the green portion of the image.) The behavior at archimedean places mimics that at the nonarchimedean places, except that the annuli invoked in the proof of Lemma \ref{lem:shape} are not round. This idea is echoed in the work of DeMarco and Faber in \cite{DeMarcoFaber1} and \cite{DeMarcoFaber2}. 
\newline

\begin{figure}
\centering
\begin{subfigure}[t]{0.5\textwidth}
	\includegraphics[scale=0.5]{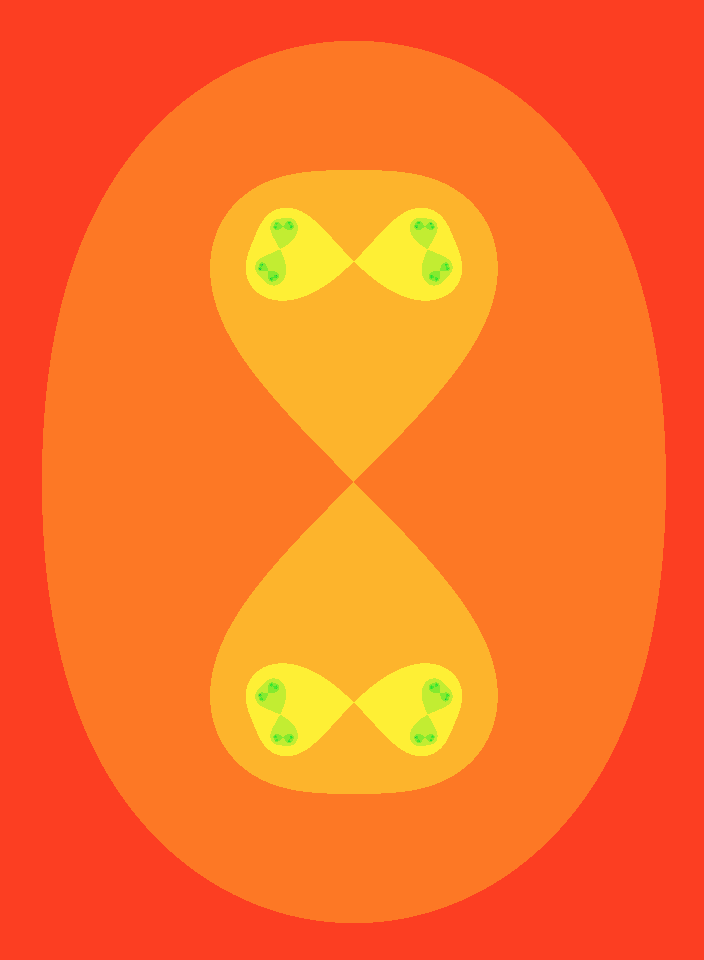}
	\caption{The equipotential curves of the map \\ $f(z)=z^2+2$, at $v=\infty$.}
\label{figure:a}
\end{subfigure}%
~
\begin{subfigure}[t]{0.5\textwidth}
		\labellist
	\small\hair 2pt
	\pinlabel $\zeta_{-\sqrt{c},|c|_v^{1/4}}$ at 240 155
	\pinlabel $\zeta_{\sqrt{c},|c|_v^{1/4}}$ at 460 180
	\pinlabel $\zeta_{0,|c|_v}$ at 342 415
	\pinlabel $0$ at 310 120
	\pinlabel $\zeta_{0,|c|_v^{1/2}}$ [bl] at 315 260
	\pinlabel $\infty$ at 312 580
	\endlabellist
	\centering
	
	\begin{tikzpicture}[->,>=stealth',auto,node distance=3cm,
	thick,main node/.style={circle,draw,font=\sffamily\Large\bfseries}]
	\node[anchor=south west,inner sep=0] at (0,-2) {\includegraphics[scale=0.45]{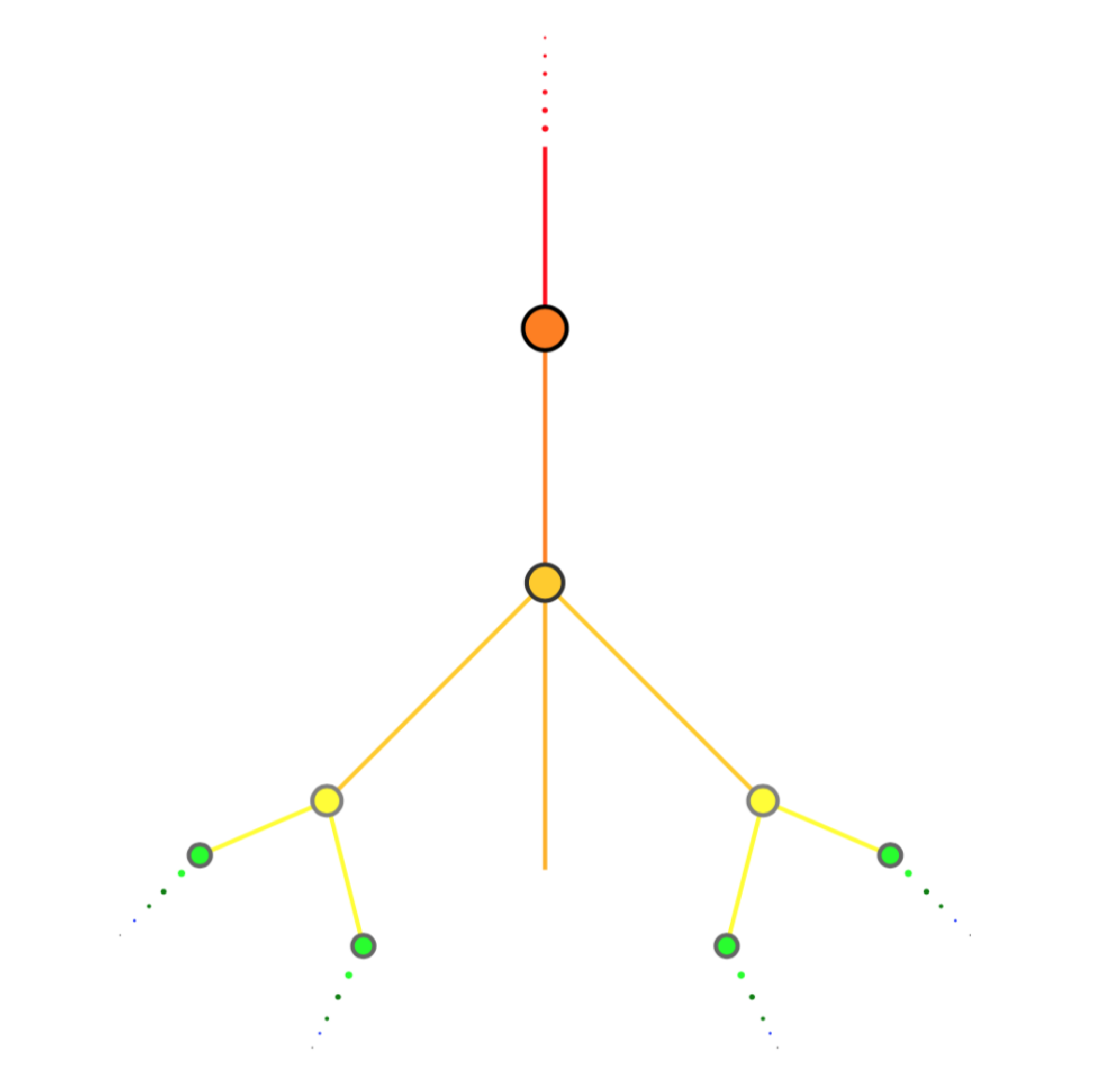}};
	\node (1) {};
	\node (A) at (2.5,.6) {};
	\node (B) at (4.3,2.4) {};
	\node (C) at (4.3, 2.4) {};
	\node (D) at (4.3, 4.5) {};
	\node (E) at (2.5,2.3) {$f$};
	\node (F) at (3.6,3.5) {$f$};
	
	\draw [->,black] (C) to [out=120,in=240] (D);
	\draw [->,black] (A) to [out=120,in=150] (B);
	\end{tikzpicture}
	
	\label{figure:BerkJulia}
	\caption{The action of a map $f(z)=z^2-c$ on $\textbf{A}_v^1$ for $v$ a place of bad reduction not dividing $2$. Here $\zeta_{a,r}$ is the Type II point corresponding to $D(a,r)$.}
	\label{figure:b}
	\end{subfigure}
\caption{
	\label{figure:1}%
	}
\end{figure}

\begin{lem}{\label{lem:shape}} Let $m\ge0$, and let $v\in M_K^0$ be such that $v\nmid d$, and $v$ is a place of bad reduction for $f$. Then \[f^{-m}(\mathcal{B}(0,|c|_v^{1/d}))=\bigcup_{i=1}^{d^m}\mathcal{B}_{m,i}\] is a union of $d^m$ closed disks in $\textbf{A}_v^1$ of radius $r_m$, where \[\log r_m=\left(\fr{1}{d}-m\fr{d-1}{d}\right)\log|c|_v.\] Moreover, for any $0\le k\le m-1$, there are $d^{m-k}-d^{m-(k+1)}=d^{m-1-k}(d-1)$ disks $\mathcal{B}_{m,j}$ lying at distance $|c|_v^{\fr{1}{d}-k\fr{d-1}{d}}$ from $\mathcal{B}_{m,i}$.\end{lem}

\begin{proof} We first claim that when $v\nmid d$ is a place of bad reduction of $f$, the greatest splitting potential of $f$ is $\fr{1}{d}\lambda_v(c)$. Indeed, $\textup{Disc}(z^d-1)=\pm d^d$, so if $\zeta_1,\zeta_2$ are $d$-th roots of unity, $v\nmid d$ implies $|\zeta_1-\zeta_2|_v=1$. Thus if $P_1,P_2$ are distinct roots of $f$, then $|P_1-P_2|_v=|c|_v^{1/d}$. It follows from Proposition \ref{prop:nonarchdiskpreimage} that \[f^{-1}(D(0,|c|_v)^-)=\bigcup_{i=1}^dD(P_i,|c|_v^{1/d})^-,\] where $P_1,\dots,P_d$ are the $d$ distinct roots of $f$. Since $f^{-1}(D(0,|c|_v))=D(0,|c|_v^{1/d})$, this proves our claim. We thus see that the fundamental annulus of $f$ is given by \[A_0=\{|c|_v^{1/d}<|z|_v<|c|_v\}.\] Since each $D_i:=D(P_i,|c|_v^{1/d})^-$ maps over $D(0,|c|_v)^-$ with degree $1$, each $D(P_i,|c|_v^{1/d})^-$ contains $d$ open disks $E_{i,j}$, $j=1,\dots,d$ mapping one-to-one onto the $D(P_i,|c|_v^{1/d})^-$ respectively. If $A_{1,i}$ is the pre-image annulus of $A_0$ given by \[A_{1,i}=D(P_i,|c|_v^{1/d})^-\smallsetminus (f^{-1}(D(0,|c|_v^{1/d})\cap D(P_i,|c|_v^{1/d})^-),\] then \[\textup{mod}(A_{1,i})=\textup{mod}(A_0)=\fr{d-1}{d}\log|c|_v.\] Hence $r_1=|c|_v^{\fr{1}{d}-\fr{d-1}{d}}$. For $1\le j_1<j_2\le d$, the disk $E_{i,j_1}$ lies at distance $|c|_v^{\fr{1}{d}-\fr{d-1}{d}}$ from $E_{i,j_2}$, and for $1\le i_1<i_2\le d$, the disk $E_{i_1,j_1}$ lies at distance $|c|_v^{1/d}$ from $E_{i_2,j_2}$. Continuing in this manner, we see that \[r_m=|c|_v^{\fr{1}{d}-\fr{m(d-1)}{d}},\] and that for each disk component $D_{m,i}$ of $f^{-m}(D(0,|c|_v^{1/d}))$, there are $d^{m-k}-d^{m-1-k}$ disks $D_{m,j}$ lying at distance $|c|_v^{\fr{1}{d}-k\fr{d-1}{d}}$ from $D_{m,i}$, for $0\le k\le m-1$. Finally, as observed previously, \cite[Lemma 9.12]{BakerRumely} implies the corresponding statement for the Berkovich disks corresponding to $D(0,|c|_v^{1/d})$ and $D_{m,i}$.
\end{proof}

We now make definitions that will be used throughout the remainder of this section. Let $v\in M_K^0$ be a place of bad reduction such that $v\nmid d$. The following objects will depend on $v$, but as $v$ is understood, we largely suppress it from the notation. Write \[f^{-m}(\mathcal{B}(0,|c|_v^{1/d}))=\bigcup_{i=1}^{d^m}\mathcal{B}_{m,i}(v)=\bigcup_{i=1}^{d^m}\mathcal{B}_{m,i}\] as in Lemma \ref{lem:shape}. For all $m\ge1$, let $D_{m,1},\dots,D_{m,d^m}$ be the classical disks corresponding to $\mathcal{B}_{m,1},\dots,\mathcal{B}_{m,d^m}$ respectively. For each $m\ge2$, fix an ordering of the $\mathcal{B}_{m,i}$ inductively, so that, for a given $m\ge1$, $\mathcal{B}_{m+1,i}\subseteq\mathcal{B}_{m,\lceil\fr{i}{d}\rceil}$. For \begin{equation}\label{eqn:k}\vec{k}=(k_{2,1},\dots,k_{2,d^2})\in(\QQ\cap[0,1])^{d^2} \text{ with }\sum_{i=1}^{d^2}k_{2,i}=1,\end{equation} define recursively for each $m\ge3$ a $d^m$-tuple \begin{equation}{\label{eqn:ktuple}}\vec{k}(m)=(k_{m,1},\dots,k_{m,d^m})\in(\QQ\cap[0,1])^{d^m}\end{equation} such that $\sum_{i=1}^{d^m}k_{m,i}=1$ and $k_{m+1,j}=\fr{1}{d}k_{m,i}$ for all $(i-1)d+1\le j\le id$. For $\vec{k}=\vec{k}(2)$ and $m\ge2$, a vector \[\vec{j}(m)=(j_{m,1},\dots,j_{m,d^m})\in(\QQ\cap[0,1])^{d^m}\] is said to be a \emph{refinement} of $\vec{k}$ if \[\sum_{l=(i-1)d^{m-2}+1}^{id^{m-2}}j_{m,l}=k_{2,i}\] for all $1\le i\le d^2$. The refinement $\vec{k}(m)$ defined in (\ref{eqn:ktuple}) is thus the particular case $j_{m,l_1}=j_{m,l_2}$ for any given $i$ and $(i-1)d^{m-2}+1\le l_1,l_2\le id^{m-2}$. For $\vec{j}(m)=(j_{m,1},\dots,j_{m,d^m})$, let $\mu_{\vec{j}(m)}$ be the unique probability measure $\nu$ on $\mathcal{E}_m:=f^{-m}(\mathcal{B}(0,|c|_v^{1/d}))$ such that $\nu(\mathcal{B}_{m,i})=j_{m,i}$, and $\nu\mid_{\mathcal{B}_{m,i}}$ is a scalar multiple of the equilibrium measure on $\mathcal{B}_{m,i}$ for all $1\le i\le d^m$. Let \[\gamma_{\vec{j}(m)}(\mathcal{E}_m):=\textup{exp}(-I(\mu_{\vec{j}(m)})).\] Note that the equilibrium measure on $\mathcal{B}_{m,i}$ is the Dirac measure at the Type II point $\zeta_{m,i}$ in $\textbf{A}_v^1$ corresponding to $D_{m,i}$. Thus \begin{equation}\label{eqn:Emeqmeasure}\mu_{\vec{j}(m)}=\sum_{i=1}^{d^m}j_{m,i}\delta_{\zeta_{m,i}},\end{equation} where $\delta_{\zeta_{m,i}}$ is the Dirac measure at $\zeta_{m,i}$.

For $\vec{k}=(k_{2,1},\dots,k_{2,d^2})\in(\QQ\cap[0,1])^{d^2}$ with $\sum_{i=1}^{d^2} k_{2,i}=1$, and $\{\vec{k}(m)\}_{m=2}^\infty$ as in (\ref{eqn:ktuple}), define $\mathcal{I}(\mu_{\vec{k}})$ as \begin{equation}\label{eqn:I}\mathcal{I}(\mu_{\vec{k}})=\lim_{m\to\infty}I(\mu_{\vec{k}(m)}).\end{equation} (Proposition \ref{prop:weightedcap} will imply this limit exists.) Let \[\gamma(\mu_{\vec{k}})=\textup{exp}(-\mathcal{I}(\mu_{\vec{k}})).\] We say a set of $n$ elements in $\mathcal{K}_v$ is \emph{$\vec{k}$-distributed} if there are $k_in$ elements in each disk component $\mathcal{B}_{2,i}$. Similarly, for $m\ge2$ and $\vec{j}(m)$ a refinement of $\vec{k}$, we say $T\subseteq\mathcal{E}_m$ of cardinality $n$ is \emph{$\vec{j}(m)$-distributed} if there are $j_{m,i}n$ elements of $T$ in $\mathcal{B}_{m,i}$ for all $i$.

\begin{prop}\label{prop:weightedcap} Let $v\in M_K^0$ be a place of bad reduction for $f$ not dividing $d$. Let $\vec{k}=(k_{2,1},\dots,k_{2,d^2})\in(\QQ\cap[0,1])^{d^2}$, with $\sum_{i=1}^{d^2}k_{2,i}=1$, and let $\{\vec{k}(m)\}_{m=2}^\infty$ be as in (\ref{eqn:ktuple}). Then for each $m\ge2$, \[\log\gamma_{\vec{k}(m)}(\mathcal{E}_m)-\log\gamma_{\vec{k}(m+1)}(\mathcal{E}_{m+1})=\left(\fr{1}{d^{m+1}}-\fr{1}{d^{m+2}}\right)\lambda_v(c).\]\end{prop}

\begin{proof} We have $\mathcal{B}_{m+1,j}\subseteq\mathcal{B}_{m,i}$ for $(i-1)d+1\le j\le id$. Write \[\mathcal{W}_{m+1,i}=\bigcup_{j=(i-1)d+1}^{id}\mathcal{B}_{m+1,j}.\] For $i=1,\dots,d^m$, let $z_i$ be the Type II point in $\textbf{A}_v^1$ corresponding to a disk $D_{m+1,j}\subseteq D_{m,i}$ (so that for each $i$, there are $d$ possible choices of $z_i$). Then \[p_{\mu_{\vec{k}(m+1)}}(z_i)=k_{m,i}\log\gamma_v(\mathcal{W}_{m+1,i})-\log\left(\prod_{j\ne i}\delta_v(z_j,z_i)^{k_{m,j}}\right),\] independently of the choice of $z_i$. Therefore \[\log\gamma_{\vec{k}(m+1)}(\mathcal{E}_{m+1})=\sum_{i=1}^{d^m}\left[k_{m,i}\log\gamma_v(\mathcal{W}_{m+1,i})-\log\prod_{j\ne i}\delta_v(z_j,z_i)^{k_{m,j}}.\right]\] On the other hand, \[\log\gamma_{\vec{k}(m)}(\mathcal{E}_m)=\sum_{i=1}^{d^m}\left[k_{m,i}\log\gamma_v(\mathcal{B}_{m,i})-\log\prod_{j\ne i}\delta_v(z_j,z_i)^{k_{m,j}}\right].\] Hence \[\log\gamma_{\vec{k}(m)}(\mathcal{E}_m)-\log\gamma_{\vec{k}(m+1)}(\mathcal{E}_{m+1})=\sum_{i=1}^{d^m}k_{m,i}\left[\log\gamma_v(\mathcal{B}_{m,i})-\log\gamma_v(\mathcal{W}_{m+1,i})\right].\] As $\sum_{i=1}^{d^m}k_{m,i}=1$, this yields \[\log\gamma_{\vec{k}(m)}(\mathcal{E}_m)-\log\gamma_{\vec{k}(m+1)}(\mathcal{E}_{m+1})=R_m,\] where $R_m$ is a constant independent of $\vec{k}$, depending only on $m$ and $v$. Since we know from \cite[Theorem 1.2]{DeMarcoRumely} and \cite[Corollary 6.26]{BakerRumely} that \[R_m=\left(\fr{1}{d^{m+1}}-\fr{1}{d^{m+2}}\right)\log|c|_v\] when $\{\vec{k}(m)\}_{m=2}^\infty=\{(1/d^m,\dots,1/d^m)\}_{m=2}^\infty$, the proposition follows.
\end{proof}

We now show that given any sufficiently large $n$ and a set $T$ of $n$ points in $\mathcal{K}_v$, and $\vec{k}$ such that $T$ is $\vec{k}$-distributed, $d_v(T)$ cannot be much larger than $\gamma(\mu_{\vec{k}})$. The requisite ``sufficiently large'' $n$ is uniform with respect to both $\vec{k}$ and $v\in M_K^0\setminus\{v\mid d\}$. Lemma \ref{lem:finitecapacity} deals with the particular case $\vec{k}=(\underbrace{1/d^2,\dots,1/d^2\,}_\text{$d^2$ times})$, $\mu_{\vec{k}}=\mu_{\mathcal{K}_v}$, leaving only the uniformity in $v$. 

\begin{lem}{\label{lem:IdealizationLemma}} Let $v\in M_K^0$ be a place of bad reduction for $f$ not dividing $d$. Let $\vec{k}=(k_{2,1},\dots,k_{2,d^2})$ be as in (\ref{eqn:k}). Let $\epsilon>0$. There exists an $N=N(\epsilon)$, independent of $\vec{k}$ and $v\nmid d$, such that for any $\vec{k}$-distributed set $T\subseteq\mathcal{K}_v$ of order $n\ge N$, we have \[\prod_{\substack{z_i\in T\\z_i\ne w}}\delta_v(z_i,w)^{1/(n-1)}\le\gamma(\mu_{\vec{k}})|c|_v^\epsilon\] for any fixed $w\in T$. In particular, \[d_v(T)\le\gamma(\mu_{\vec{k}})|c|_v^\epsilon.\] 
\end{lem}

\begin{proof} Let $\epsilon'>0$, let $m\ge2$, and fix $\vec{k}=\vec{k}(2)$. Let $\vec{j}(m)=(j_{m,1},\dots,j_{m,d^m})$ be a refinement of $\vec{k}(2)$. We claim that there is an $N=N(\epsilon',m)$ such that if $T\subseteq\mathcal{E}_m$ is a set of $n\ge N$ elements $z_1,\dots,z_n\in\mathcal{E}_m$, and $T$ is $\vec{j}(m)$-distributed, then \begin{equation}{\label{ineq:levelmconv}}d_v(T)\le|c|_v^{\epsilon'}\gamma_{\vec{j}(m)}(\mathcal{E}_m).\end{equation}
	
Indeed, suppose $T\subseteq\mathcal{E}_m$ is $\vec{j}(m)$-distributed, and let $r_m$ be as in Lemma \ref{lem:shape}. By the description of $\mu_{\vec{j}(m)}$ given in (\ref{eqn:Emeqmeasure}), we have that for any $w\in T\cap\mathcal{B}_{m,l}$, \begin{equation*}\begin{split}\textup{exp}(-p_{\mu_{\vec{j}(m)}}(w))&\ge\left(\prod_{\substack{z_i\in T\\z_i\ne w}}\delta_v(z_i,w)^{1/(n-1)}\right)\left(\tfrac{1}{r_m}\right)^{\fr{(j_{m,l})n-1}{n-1}}(r_m)^{j_{m,l}} \\&=\left(\prod_{\substack{z_i\in T\\z_i\ne w}}\delta_v(z_i,w)^{1/(n-1)}\right)(r_m)^{\fr{1-j_{m,l}}{n-1}}\\&\ge\left(\prod_{\substack{z_i\in T\\z_i\ne w}}\delta_v(z_i,w)^{1/(n-1)}\right)(r_m)^{\fr{1}{n-1}},\end{split}\end{equation*} where the final inequality follows from the fact that $r_m<1$. It follows from the description of $r_m$ in Lemma \ref{lem:shape} that given an $\epsilon'>0$, there is an $N=N(\epsilon',m)$ such that if $|T|\ge N$, then \[|c|_v^{-\epsilon'}\prod_{\substack{z_i\in T\\z_i\ne w}}\delta_v(z_i,w)^{1/(n-1)}\le\gamma_{\vec{j}(m)}(\mathcal{E}_m),\] for an $\epsilon'>0$ independent of $v$ and $\vec{j}(m)$ (as well as the $\vec{k}$ of which $\vec{j}(m)$ is a refinement). This proves the claim. 
	
Let $\vec{k}(m)$ be as in (\ref{eqn:ktuple}). Then by symmetry and the uniqueness of the equilibrium measure on each $\mathcal{B}_{2,i}\cap\mathcal{E}_m$, we have \begin{equation}\label{ineq:jvsk}\gamma_{\vec{j}(m)}(\mathcal{E}_m)\le\gamma_{\vec{k}(m)}(\mathcal{E}_m)\end{equation} for any refinement $\vec{j}(m)$ of $\vec{k}$. The lemma then follows from Proposition \ref{prop:weightedcap} combined with (\ref{ineq:levelmconv}) and (\ref{ineq:jvsk}).\end{proof}

An important special case of Lemma \ref{lem:IdealizationLemma} is the case $\vec{k}(2)=(1/d^2,\dots,1/d^2)$. We record it here for future reference.

\begin{lem}{\label{lem:finitecapacity}} Let $v\in M_K^0$ be a place of bad reduction for $f$ not dividing $d$. There is an integer $N=N(\epsilon)$ such that any set $T$ of $n\ge N$ points in $\mathcal{K}_v$ has \[\log\prod_{\substack{z_i\in T\\z_i\ne w}}\delta_v(z_i,w)^{1/(n-1)}\le\epsilon\lambda_v(c)\] for any $w\in T$, and thus, \[\log d_v(T)\le\epsilon\lambda_v(c).\]\end{lem}

\begin{proof} The statement is clear when $v\in M_K^0$ is a place of good reduction. Let $v\nmid d$ be a place of bad reduction, and let $\vec{k}=(1/d^2,\dots,1/d^2)$ be as in (\ref{eqn:k}). By \cite[Theorem 1.2]{DeMarcoRumely} and \cite[Corollary 6.26]{BakerRumely}, we have $\gamma(\mu_{\vec{k}})=1$. Moreover, Proposition \ref{prop:weightedcap} implies that $\gamma(\mu_{\vec{k}'})\le\gamma(\mu_{\vec{k}})$ for any other $\vec{k}'=(k_{2,1},\dots,k_{2,d^2})$ as in (\ref{eqn:k}). Lemma \ref{lem:IdealizationLemma} thus completes the proof.\end{proof}

\begin{prop}{\label{prop:equidistribute}} Let $\epsilon_1>0$, and let $v\in M_K^0$ be a place of bad reduction for $f$ such that $v\nmid d$. Write \[\mathcal{E}_2=f^{-2}(\mathcal{B}(0,|c|_v^{1/d}))=\bigcup_{i=1}^{d^2}\mathcal{B}_{2,i},\] where each $\mathcal{B}_{2,i}$ is a closed disk in $\textbf{A}_v^1$ about a root of $f^2$. Let $T\subseteq\mathcal{K}_v$, with $|T|=n$. For each $1\le i\le d^2$, let $b_i=|T\cap\mathcal{B}_{2,i}|$. There exist an $\epsilon_2=\epsilon_2(\epsilon_1)>0$ and an $N=N(\epsilon_1)$ such that if $n\ge N$ and some $b_i$ fails to satisfy \begin{equation*}\fr{1-\epsilon_1}{d^2}n< b_i< \fr{1+\epsilon_1}{d^2}n,\end{equation*} then $d_v(T)\le|c|_v^{-\epsilon_2}$.\end{prop}

\begin{proof} Without loss of generality, suppose $\epsilon_1\in\QQ$. Suppose first that there is some $1\le i\le d^2$ such that $|\mathcal{B}_{2,i}\cap T|<\fr{1-\epsilon_1}{d^2}n$. Let $\mathcal{M}$ be the set of $d^2$-tuples $\vec{k}(2)$ as in (\ref{eqn:ktuple}) such that either $k_{2,i}\le\fr{1-\epsilon_1}{d^2}$ for some $1\le i\le d^2$, or $k_{2,i}\ge\fr{1+\epsilon_1}{d^2}$ for some $1\le i\le d^2$. Then $\mathcal{M}$ is compact. The function $\mathcal{I}(\mu_{\vec{k}}):\mathcal{M}\to\mathbb{R}$ defined in (\ref{eqn:I}) is clearly continuous. It therefore attains a minimum on the compact set $\mathcal{M}$. Suppose this occurs at $\vec{k}_0$. 
	
Viewed as Borel measures on $\mathcal{E}_2$, the $\mu_{\vec{k}(m)}$ must have a subsequence $\mu_{\vec{k}(m_l)}$ that is weak$^*$-convergent on $\mathcal{E}_2$, by \cite[Theorem A.10]{BakerRumely}. Denote the limit by $\mu_{\vec{k}}^*$. It is clear that $\text{supp}(\mu_{\vec{k}}^*)\subseteq\mathcal{K}_v$. We want to show: 
	
	\textbf{Claim}: For each $\vec{k}=\vec{k}(2)$, \[\mathcal{I}(\mu_{\vec{k}})=I(\mu_{\vec{k}}^*).\] Assuming this claim, it follows from Proposition \ref{prop:weightedcap} that \begin{equation}\label{eqn:bdabove}\textup{exp}(-I(\mu_{\vec{k}_0}^*))=\gamma(\mu_{\vec{k}_0})=|c|_v^{-\epsilon'}\end{equation} for some $\epsilon'=\epsilon'(\epsilon_1)$. By \cite[Theorem 10.91(D)]{BakerRumely}, we have $\gamma_v(\mathcal{K}_v)=1$; thus, \[\textup{exp}(-I(\mu_{\vec{k}}^*))\le\textup{exp}(-I(\mu_{\vec{k}_0}^*))<\gamma_v(\mathcal{K}_v)=1,\] where the second inequality follows from the uniqueness of the equilibrium measure on $\mathcal{K}_v$. The proof of the claim is then exactly parallel to that of \cite[Corollary 6.9]{BakerRumely}. 
	
	Proof of claim: It suffices to show that \[\mathcal{I}(\mu_{\vec{k}})\ge I(\mu_{\vec{k}}^*).\] As it is easily seen that \[I(\mu_{\vec{k}(m)})\le I(\mu_{\vec{k}}^*)\] for all $m\ge2$, the desired claim will then follow immediately.
	
	Let $\mathcal{C}(\mathcal{E}_2\times\mathcal{E}_2)$ denote the space of continuous real-valued functions on $\mathcal{E}_2\times\mathcal{E}_2$. Since $\mathcal{E}_2$ is a compact subset of $\mathbf{A}_v^1$, we know that $-\log\delta_v(x,y)$ is bounded below on $\mathcal{E}_2\times\mathcal{E}_2$ by some constant $-M$. By \cite[Proposition A.3]{BakerRumely}, for any probability measure $\mu$ on $\mathcal{E}_2$, we thus have \[I(\mu)=\sup_{\substack{g\in\mathcal{C}(\mathcal{E}_2\times\mathcal{E}_2)\\-M\le g\le-\log\delta_v(x,y)}}\iint_{\mathcal{E}_2\times\mathcal{E}_2}g(x,y)d\mu(x)d\mu(y).\] By \cite[Lemma 6.5]{BakerRumely}, for each $g\in\mathcal{C}(\mathcal{E}_2\times\mathcal{E}_2)$, \[\iint_{\mathcal{E}_2\times\mathcal{E}_2}g(x,y)d\mu_{\vec{k}}^*(x)d\mu_{\vec{k}}^*(y)=\lim_{l\to\infty}\iint_{\mathcal{E}_2\times\mathcal{E}_2}g(x,y)d\mu_{\vec{k}(m_l)}(x)d\mu_{\vec{k}(m_l)}(y).\]	When $g(x,y)\le-\log\delta_v(x,y)$ on $\mathcal{E}_2\times\mathcal{E}_2$, for each $m_l$ we have \[\iint_{\mathcal{E}_2\times\mathcal{E}_2}g(x,y)d\mu_{\vec{k}(m_l)}(x)d\mu_{\vec{k}(m_l)}(y)\le I(\mu_{\vec{k}(m_l)}).\] It follows that for any such $g$, \begin{equation*}\begin{split}\lim_{l\to\infty}I(\mu_{\vec{k}(m_l)}(\mathcal{E}_{m_l}))&\ge\lim_{l\to\infty}\iint_{\mathcal{E}_2\times\mathcal{E}_2}g(x,y)d\mu_{\vec{k}(m_l)}(x)d\mu_{\vec{k}(m_l)}(y)\\&=\iint_{\mathcal{K}_v\times\mathcal{K}_v}g(x,y)d\mu_{\vec{k}}^*(x)d\mu_{\vec{k}}^*(y).\end{split}\end{equation*} Taking the supremum over all such $g\in\mathcal{C}(\mathcal{E}_2\times\mathcal{E}_2)$ yields \[\lim_{l\to\infty}I(\mu_{\vec{k}(m_l)}(\mathcal{E}_{m_l}))\ge\sup_g\iint_{\mathcal{K}_v\times\mathcal{K}_v}g(x,y)d\mu_{\vec{k}}^*(x)d\mu_{\vec{k}}^*(y)=I(\mu_{\vec{k}}^*)\] as claimed.
	
	From Lemma \ref{lem:IdealizationLemma} and the above claim, we know that there is an $N=N(\epsilon')$ such that if $|T|=n\ge N$ is $\vec{k}$-distributed, then \begin{equation}\label{eqn:smalldiam} d_v(T)\le\gamma(\mu_{\vec{k}})|c|_v^{\epsilon'/2}=\textup{exp}(-I(\mu_{\vec{k}}^*))|c|_v^{\epsilon'/2}.\end{equation} As, by (\ref{eqn:bdabove}) and the observation immediately succeeding it,  \begin{equation}\label{eqn:weaklimit}\textup{exp}(-I(\mu_{\vec{k}}^*))\le\textup{exp}(-I(\mu_{\vec{k}_0}^*))=|c|_v^{-\epsilon'}\end{equation} for $\epsilon'=\epsilon'(\epsilon_1)>0$, the proof is then completed by taking $\epsilon_2=\epsilon'/2$ and combining (\ref{eqn:smalldiam}) and (\ref{eqn:weaklimit}).\end{proof}

\begin{definition*} For $v\in M_K^0$, $\epsilon>0$, and a finite set $T\subseteq\mathcal{K}_v$ of cardinality $n$, we say that $T$ is \emph{$\epsilon$-equidistributed (at $v$)} if there are $b_i$ elements of $T$ in each $\mathcal{B}_{2,i}$, where $b_i$ satisfies \[\fr{1-\epsilon}{d^2}n<b_i<\fr{1+\epsilon}{d^2}n.\]\end{definition*}

From Lemma \ref{lem:finitecapacity} and Proposition \ref{prop:equidistribute}, we obtain a global quantitative equidistribution statement: given $\epsilon>0$, any sufficiently large set of $K$-rational preperiodic points of $f$ is $\epsilon$-equidistributed at ``most" places of bad reduction. 

\begin{cor}{\label{cor:bucketequid}} Let $\epsilon>0$, and let $0<\delta<1$. There are real numbers $M(\epsilon,\delta),\xi(\epsilon,\delta)$, and $N(\epsilon,\delta)$ such that if $h(c)\ge M$, \[\sum_{v\in M_K^0\setminus\{v\mid d\}}r_v\lambda_v(c)\ge(1-\xi)h(c),\] and $T$ is a set of $K$-rational preperiodic points of $f$ with $|T|=n\ge N$, then there is an $\delta$-slice $S$ of bad places $v\in M_K^0\setminus\{v\mid d\}$ such that for each $v\in S$, the set $T$ is $\epsilon$-equidistributed at $v$. 
\end{cor}

\begin{proof} Let $T$ be a set of $K$-rational preperiodic points of $f$, with $|T|=n$. Let $S_1$ be the set of places of bad reduction for $f$ not dividing $d$ at which $T$ is $\epsilon$-equidistributed, and let $S_2$ be the set of places of bad reduction for $f$ not dividing $d$ at which $T$ fails to be $\epsilon$-equidistributed. Suppose that \[\sum_{v\in S_1}r_v\lambda_v(c)\le\delta_0\sum_{v\in M_K^0\setminus\{v\mid d\}}r_v\lambda_v(c)\] for $\delta_0<\delta$. We will show that this leads to a contradiction when $n$ is sufficiently large. Let $0<\xi\le1$ be such that \[\sum_{v\in M_K^0\setminus\{v\mid d\}}r_v\lambda_v(c)\ge(1-\xi)h(c).\] For $v\in M_K^\infty\cup\{v\mid d\}$, \cite[Lemmas 3 and 6]{Ingram:mincanht} yields \[\log d_v(T)\le\fr{1}{d}\lambda_v(c)+\log4.\] Thus, when $h(c)\gg_\xi1$, we have \begin{equation}{\label{ineq:2xi}}\sum_{M_K^\infty\cup\{v\mid d\}}r_v\log d_v(T)\le\fr{2\xi}{d}\sum_{v\in M_K^0\setminus\{v\mid d\}}r_v\lambda_v(c).\end{equation} Let $\epsilon_1>0$. By Lemma \ref{lem:finitecapacity}, there is an $N_1=N_1(\epsilon_1)$ such that if $n\ge N_1=N_1(\epsilon_1)$, then $d_v(T)\le|c|_v^{\epsilon_1}$ for any $v\in M_K^0\setminus\{v\mid d\}$. Assume that $n\ge N_1$, and $h(c)\gg_\xi1$, so that (\ref{ineq:2xi}) holds. Then \begin{equation}{\label{ineq:goodeq}}\sum_{M_K^\infty\cup\{v\mid d\}}r_v\log d_v(T)+\sum_{v\in S_1}r_v\log d_v(T)<\left(\fr{2\xi}{d}+\epsilon_1\delta_0\right)\sum_{v\in M_K^0\setminus\{v\mid d\}}r_v\lambda_v(c).\end{equation} Moreover, by Proposition \ref{prop:equidistribute}, there is an $\epsilon_2(\epsilon)>0$ and an $N_2(\epsilon)$ such that if $n\ge N_2$ and $T$ fails to be $\epsilon$-equidistributed at some $v\in M_K^0\setminus\{v\mid d\}$, then $d_v(T)\le|c|_v^{-\epsilon_2}$. Suppose $n\ge\max\{N_1,N_2\}$. Then \begin{equation}{\label{ineq:badeq}}\begin{split}\sum_{v\in S_2}r_v\log d_v(T)&\le-\epsilon_2\sum_{v\in S_2}r_v\lambda_v(c)\\&\le-\epsilon_2(1-\delta_0)\sum_{v\in M_K^0\setminus\{v\mid d\}}r_v\lambda_v(c).\end{split}\end{equation} Observing that for $v\in M_K^0$ a place of good reduction, $\log d_v(T)\le 0$, combining (\ref{ineq:goodeq}) and (\ref{ineq:badeq}) yields \[\sum_{v\in M_K}r_v\log d_v(T)\le \left(\fr{2\xi}{d}+\epsilon_1\delta_0-\epsilon_2(1-\delta_0)\right)\sum_{v\in M_K^0\setminus\{v\mid d\}}r_v\lambda_v(c).\] If $\epsilon_1,\xi$ are chosen to be sufficiently small (depending on $\delta_0$ and $\epsilon_2$), then the right-hand side is strictly less than zero, contradicting the product formula. Since $\delta_0<\delta$, this proves the desired claim.\end{proof} 

\section{An easier case of uniformity}{\label{section:easycase}}

This section is devoted to proving that uniform boundedness holds across any family of polynomials having a subset of places of bad reduction of $f$ contributing at least some fixed proportion to $h(c)$. From this fact, one sees that the difficulty in proving Theorem \ref{thm:UBCunicrit} lies in dealing with parameters $c$ for which the height contribution from each prime $\mathfrak{p}$ with $v_\mathfrak{p}(c)<0$ is arbitrarily small. In the proof of Theorem \ref{thm:UBCunicrit}, we will need to assume the height contribution is small at the archimedean places, and at the places dividing $d$. Proposition \ref{prop:xi} suffices to handle maps for which this does not hold. It is latent in the proofs of \cite[Main Theorem]{Benedetto}, \cite[Theorem 1]{Ingram:mincanht}, and \cite[Theorem 1.2]{Looper:mincanht} (which in turn are evocative of Lemmas 3 and 4 of \cite{Silverman:mincanht}).

\begin{prop}{\label{prop:xi}} Fix $d\ge2$ and $K$ a product formula field. Let $f(z)=z^d+c\in K[z]$, and let $\xi>0$. Let $s$ be a positive integer, and let $S$ be any nonempty set of places of $K$ with $|S|\le s$. If $K$ is a function field, assume $h(c)>0$. There is an $N=N(\xi,s,K)$ such that if $f$ satisfies \begin{equation*}\sum_{v\in S}r_v\lambda_v(c)\ge\xi h(c),\end{equation*} then there are at most $N$ preperiodic points of $f$ contained in $K$.\end{prop}

For a real number $\eta$, let $(\eta)_K=\eta$ if $K$ is a number field, and let $(\eta)_K=0$ if $K$ is a function field. Let $\eta_v=\eta$ if $v\in M_K^\infty$, and let $\eta_v=1$ otherwise. In our proof of Proposition \ref{prop:xi}, we will use a result governing the distance between a point of low local canonical height and the nearest root of $f^3$.

\begin{prop}\cite[Proofs of Corollary 3.4, Proposition 4.3]{Looper:mincanht}\label{prop:excsmush} Let $f(z)=z^d+c\in K[z]$ with $d\ge2$, and let $v\in M_K$. Let $\alpha\in\CC_v$. There is a constant $\eta$ depending only on $d$ such that $\alpha\in D(0,2_v|c|_v^{1/d})$ implies that every $y\in f^{-3}(\alpha)$ satisfies \[\min_{\beta\in f^{-3}(0)}\log|y-\beta|_v\le \left(\fr{3}{d}-2\right)\lambda_v(c)+(\eta)_K.\]
\end{prop}

\begin{proof}[Proof of Proposition \ref{prop:xi}] Suppose $f$ satisfies \begin{equation*}\sum_{v\in S}r_v\lambda_v(c)\ge\xi h(c)\end{equation*} for some nonempty set $S$ of places of $K$ with $|S|\le s$. We must have \[r_{v_0}\lambda_{v_0}(c)\ge\fr{\xi}{s}h(c)\] for some $v_0\in S$. Let $S_0=M_K^\infty\cup\{v\mid d\}\cup\{v_0\}$, with $|S_0|=s_0$. Then a fortiori, \begin{equation}\label{ineq:S0}\sum_{v\in S_0}r_v\lambda_v(c)\ge\fr{\xi}{s}h(c).\end{equation} Let $\epsilon>0$, and let $T$ be the set of $K$-rational preperiodic points of $f$, with $|T|=n$. Let $T'\subseteq T$ be such that for each $v\in S_0$, $T'$ is contained in a single disk of radius at most $e^{(\eta)_K}|c|_v^{\fr{3}{d}-2}$ in $\CC_v$. Proposition \ref{prop:excsmush} and the pigeonhole principle imply that if $n\gg_{s_0}1$, then there exists such a nonempty $T'$ satisfying $|T'|\ge n/d^{3s_0}$. By Lemma \ref{lem:finitecapacity}, there is an $N=N(\epsilon)$ such that if $|T'|\ge N$, then \[d_v(T')\le|c|_v^\epsilon\] for all $v\in M_K\setminus S_0$. This yields \begin{equation}\label{ineq:T'crush}\sum_{v\in M_K}r_v\log d_v(T')\le \sum_{v\in S_0}\left(\left(\fr{3}{d}-2\right)r_v\lambda_v(c)+(\eta)_K\right)+\sum_{v\in M_K\setminus S_0}r_v\epsilon\lambda_v(c).\end{equation} Suppose $K$ is a number field. If $\epsilon$ is chosen to be sufficiently small, and $h(c)\gg_{\xi,s,s_0,\epsilon}1$, then (\ref{ineq:S0}) and (\ref{ineq:T'crush}) give \begin{equation*}\begin{split}\sum_{v\in M_K}r_v\log d_v(T')&\le \sum_{v\in S_0}\left(\left(\fr{3}{d}-2\right)r_v\lambda_v(c)+\eta\right)+\sum_{v\in M_K\setminus S_0}r_v\epsilon\lambda_v(c)\\&\le \left(\fr{3}{d}-2+\epsilon\right)\fr{\xi}{s}h(c)+\sum_{v\in M_K\setminus S_0}r_v\epsilon\lambda_v(c)\\&\le\left(\fr{3}{d}-2+\epsilon\right)\fr{\xi}{s}h(c) +\left(1-\left(\xi-\fr{\xi}{s}\right)\right)\epsilon h(c)\\&<0,\end{split}\end{equation*} contradicting the product formula. Northcott's Theorem then accounts for the finitely many remaining values of $c\in K$. If $K$ is a function field, then $S_0=\{v_0\}$ and so (\ref{ineq:T'crush}) becomes \begin{equation}\label{ineq:fnfldxi}\begin{split}\sum_{v\in M_K}r_v\log d_v(T')&\le \left(\fr{3}{d}-2\right)r_{v_0}\lambda_{v_0}(c)+\sum_{v\ne v_0\in M_K}r_v\epsilon\lambda_v(c)\\&\le\left(\fr{3}{d}-2\right)\fr{\xi}{s}h(c)+\left(1-\left(\xi-\fr{\xi}{s}\right)\right)\epsilon h(c).\end{split}\end{equation} Since $h(c)>0$ by assumption, the right-hand side of (\ref{ineq:fnfldxi}) is strictly less than $0$ when $\epsilon$ is sufficiently small, contradicting the product formula. The proof is completed by noting that $|T'|\ge N=N(\epsilon)$ when $n\gg_{s_0,\epsilon}1$.\end{proof} 

\section{Adelic properties of differences of preperiodic points}{\label{section:adelic}}

Throughout this section, we fix $d\ge2$, $K$ a product formula field, and $f(z)=z^d+c\in K[z]$. Our goal is to show that when $f$ has a large number of $K$-rational preperiodic points, elements of the form $p_i-p_j$ for preperiodic $p_i,p_j\in K$ typically have their prime support mostly contained within the set of places of bad reduction. This can be viewed as a dynamical analogue of what one observes in the setting of groups, where roots of unity, as well as torsion points on elliptic curves, stay distinct modulo primes of good reduction not dividing their orders. We also remark that this phenomenon can be shown more easily in the case of periodic points, by leveraging for example \cite[Theorem 6.3]{MortonSilverman:units} or \cite[Lemma 1]{Narkiewicz} in conjunction with Lemma \ref{lem:finitecapacity}. 

Corollary \ref{cor:adgoodbars} is formulated in terms of how well behaved the differences of preperiodic points must be for our purposes in the proof of Theorem \ref{thm:UBCunicrit}. In particular, the numbers $1/600$ and $1/800$, which we employ for concreteness, may each be replaced by any other positive real numbers.

\begin{prop}{\label{prop:adgoodbars}} Let $\epsilon>0$. Let $S_1$ be the set of places of bad reduction for $f$ not dividing $d$. Fix $p_1\in K$ a preperiodic point of $f$, and suppose \[\mathcal{T}:=\{p_j-p_1: p_j\in K\text{ a preperiodic point of }f\}\] has cardinality $n$. There exist real numbers $1>\xi=\xi(\epsilon)>0$ and $M=M(\epsilon)$, and an integer $N=N(\epsilon)$, such that if $n\ge N$, $h(c)\ge M$ and \[\sum_{v\in M_K^\infty\cup\{v\mid d\}}r_v\lambda_v(c)\le\xi h(c),\] 
then at least $(1-\epsilon)n$ elements $a_j\in\mathcal{T}$ satisfy \[\sum_{v\in S_1} r_v\log|a_j|_v\le\fr{1}{800}h(c).\]\end{prop}

\begin{proof} Let $\epsilon>0$, and suppose $f(z)=z^d+c\in K[z]$ has the property \[\sum_{v\in M_K^\infty\cup\{v\mid d\}}r_v\lambda_v(c)\le\xi h(c)\] for some $1>\xi>0$. By Lemma \ref{lem:finitecapacity}, there is an $N=N(\epsilon)$ such that if $n\ge N$ and $v\in S_1$, then \[\fr{1}{n}\sum_{a_j\in\mathcal{T}}\log|a_j|_v\le\fr{\epsilon}{900}\lambda_v(c),\] and hence \begin{equation}{\label{eqn:avgbar}} \fr{1}{n}\sum_{a_j\in\mathcal{T}}\sum_{v\in S_1}r_v\log|a_j|_v\le\fr{\epsilon}{900}\sum_{v\in S_1}r_v\lambda_v(c).\end{equation} Assume $n\ge N$. On the other hand, we also claim that for any $a_j\in\mathcal{T}$, \begin{equation}{\label{eqn:barlowerbound}} \sum_{v\in S_1}r_v\log|a_j|_v\ge-\log4-\fr{\xi}{d}h(c).\end{equation} Indeed, \[\sum_{v\in M_K^\infty\cup\{v\mid d\}}r_v\lambda_v(c)\le\xi h(c),\] so \[\sum_{v\in M_K^\infty\cup\{v\mid d\}}r_v\log|a_j|_v\le\sum_{v\in M_K^\infty\cup\{v\mid d\}}r_v\left(\fr{1}{d}\lambda_v(c)+\log4\right)\le\log4+\fr{\xi}{d}h(c).\] From (\ref{eqn:barlowerbound}), one sees that if $h(c)\gg_\xi1$, then \begin{equation}{\label{eqn:xi}}\sum_{v\in S_1}r_v\log|a_j|_v\ge-\xi h(c).\end{equation} It follows that at least $(1-\epsilon)n$ of the elements $a_j\in\mathcal{T}$ must satisfy \[\sum_{v\in S_1}r_v\log|a_j|_v\le\fr{1}{800}\sum_{v\in S_1}r_v\lambda_v(c);\] otherwise, by (\ref{eqn:xi}), we see that for all $h(c)\gg_\xi1$ and any sufficiently small choice of $\xi=\xi(\epsilon)$, we would have \[\fr{1}{n}\sum_{a_j\in\mathcal{T}}\sum_{v\in S_1}r_v\log|a_j|_v>\sum_{v\in S_1}r_v\lambda_v(c)\left(\fr{\epsilon}{800}-\xi(1-\epsilon)\right)>\fr{\epsilon}{900}\sum_{v\in S_1}r_v\lambda_v(c),\] contradicting (\ref{eqn:avgbar}). 
\end{proof}

Let $S_{2,1}$ be the set of places of good reduction for $f$, and let $S_{2,2}=M_K^\infty\cup\{v\mid d\}$. 

\begin{definition*} We say that $a\in K^*$ is \emph{adelically good} if \begin{equation*}\sum_{v=\mathfrak{p}\in S_{2,1}} v_\mathfrak{p}(a)N_\mathfrak{p}\le\fr{1}{600}h(c),\end{equation*} and \begin{equation}\label{eqn:archplaces}\sum_{v\in S_{2,2}}r_v\log|a|_v\ge-\fr{1}{800}h(c).\end{equation}
\end{definition*} By the product formula, Proposition \ref{prop:adgoodbars} immediately implies: \begin{cor}{\label{cor:adgoodbars}} Let $\epsilon>0$. Fix $p_1\in K$ a preperiodic point of $f$, and suppose \[\mathcal{T}:=\{p_j-p_1:p_j\in K\text{ is a preperiodic point of }f\}\] has cardinality $n$. There exist an integer $N=N(\epsilon)$ and real numbers $1>\xi=\xi(\epsilon)>0$, $M=M(\epsilon)$ such that if $h(c)\ge M$, $n\ge N$, and \[\sum_{v\in S_{2,2}}r_v\lambda_v(c)\le\xi h(c),\] then at least $(1-\epsilon)n$ elements of $\mathcal{T}$ are adelically good.\end{cor}

\section{The quadratic case}{\label{section:d2}}

Let $K$ be a number field or a one-variable function field of characteristic zero. We will be applying Conjecture \ref{conj:nconj} to $6$-tuples formed using a fixed preperiodic point $p_1\in K$ of $f(z)=z^2+c\in K[z]$, as well as four other hypothetical $K$-rational preperiodic points $p_2,\dots,p_5$ of $f$. We will refer to these tuples as hexagons, as illustrated in Figure \ref{figure:hexagon}.

\begin{definition*} Let $p_1\in K$ be preperiodic under $f(z)=z^2+c\in K[z]$. A \emph{hexagon with basepoint $p_1$} is an element \[w=(x_0,x_1,x_2,x_3,x_4,x_5)\in K^6\] such that $x_0=p_2-p_1$, $x_1=p_1-p_3$, $x_2=-p_1-p_4$, $x_3=p_5+p_1$, $x_4=p_3-p_5$, and $x_5=p_4-p_2$ for $K$-rational preperiodic points $p_2,p_3,p_4,p_5$ of $f$. \end{definition*}

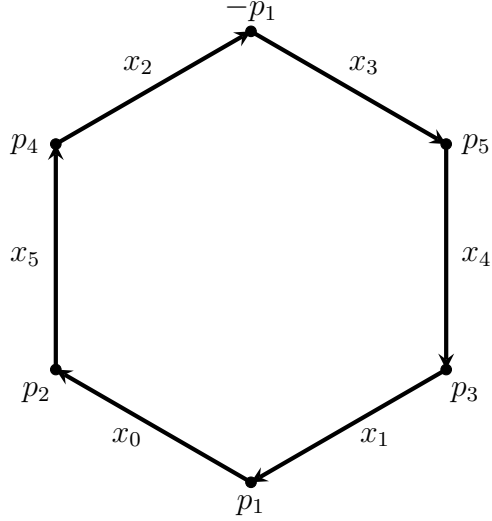
\begin{figure}[H]
	\begin{tikzpicture}[auto,scale=0.15,>=stealth]
	
	\draw[->, ultra thick, black](17.3,10) -- (0,0);
	\draw[->, ultra thick, black](17.3,30) -- (17.3,10);
	\draw[->, ultra thick, black](0,40) -- (17.3,30);
	\draw[->, ultra thick, black](-17.3,30) -- (0,40);
	\draw[->, ultra thick, black](-17.3,10) -- (-17.3,30);
	\draw[->, ultra thick, black](0,0) -- (-17.3,10);
	
	\draw (-20,20) node {$x_5$};
	\draw (20,20) node {$x_4$};
	\draw (-11,4) node {$x_0$};
	\draw (-20,30) node {$p_4$};
	\draw (20,30) node {$p_5$};
	\draw (0,-2) node {$p_1$};
	\draw (-19,8) node {$p_2$};
	\draw (11,4) node {$x_1$};
	\draw (19,8) node {$p_3$};
	\draw (-10,37) node {$x_2$};
	\draw (10,37) node {$x_3$};
	\draw (0,42) node {$-p_1$};
	\node at (0,0) {\pgfuseplotmark{*}};
	\node at (17.3,10) {\pgfuseplotmark{*}};
	\node at (-17.3,10) {\pgfuseplotmark{*}};
	\node at (0,40) {\pgfuseplotmark{*}};
	\node at (17.3,30) {\pgfuseplotmark{*}};
	\node at (-17.3,30) {\pgfuseplotmark{*}};
	\end{tikzpicture}
	\caption{A hexagon with basepoint $p_1$.}
	\label{figure:hexagon}
\end{figure}

\begin{proof}[Proof of Theorem \ref{thm:UBCunicrit} when $d=2$] Let $\epsilon>0$. Let $\xi=\xi(\epsilon)$, $N=N(\epsilon)$, and $M=M(\epsilon)$ be as in Corollary \ref{cor:adgoodbars}, and assume $f(z)=z^2+c\in K[z]$ satisfies \[\sum_{v\in M_K^0\setminus\{v\mid 2\}}r_v\lambda_v(c)\ge(1-\xi)h(c),\] $h(c)\ge M$, and that $f$ has $n\ge N$ preperiodic points in $K$. Suppose $p_1\in K$ is a preperiodic point of $f$. Let $[X_0:\cdots:X_5]$ be the standard homogeneous coordinates on $\mathbb{P}^5(K)$, and let $\mathcal{Z}=\mathcal{Z}(K,\epsilon,6)$ be as in Conjecture \ref{conj:nconj}. The variety $\mathcal{Z}$ is the zero locus of a finite set of nonzero homogeneous polynomials: by choosing only one of these, thereby possibly enlarging $\mathcal{Z}$, we can assume $\mathcal{Z}$ is defined by a single homogeneous polynomial in $G(X_0,\dots,X_5)\in\overline{K}[X_0,\dots,X_5]$. For any hexagon \[w=(x_0,x_1,x_2,x_3,x_4,x_5)=(p_2-p_1,p_1-p_3,-p_1-p_4,p_5+p_1,p_3-p_5,p_4-p_2)\] with basepoint $p_1$, we have $x_0+x_1+\dots+x_5=0$. Thus, letting $P=(x_0,x_1,x_2,x_3,x_4,x_5)\in\mathbb{P}^5(K)$, we have $P\in\mathcal{H}$, where $\mathcal{H}$ is the hyperplane in $\mathbb{P}^5(K)$ given by $X_0+\dots+X_5=0$. We would like to produce a hexagon $w=(x_0,x_1,x_2,x_3,x_4,x_5)$ such that the associated point  $P\in\mathbb{P}^5(K)$ satisfies $P\in\mathcal{H}\setminus\mathcal{Z}$. Letting $X_0=P_2-p_1$, $X_1=p_1-P_3$, $X_2=-p_1-P_4$, $X_3=P_5+p_1$, $X_4=P_3-P_5$, and $X_5=P_4-P_2$, we can write $G(X_0,\dots,X_5)=G(P_2,P_3,P_4,P_5)$ as \begin{equation}\label{eqn:leastmonomials}G(P_2,P_3,P_4,P_5)=\sum_{l=1}^ra_lP_2^{i_{l,2}}P_3^{i_{l,3}}P_4^{i_{l,4}}P_5^{i_{l,5}}\in\overline{K}[P_2,P_3,P_4,P_5].\end{equation}
	
For $\vec{m}=(m_3,m_4,m_5)\in\mathbb{N}^3$, let \[G_{\vec{m}}=\sum_{i}b_iP_2^{m_{2,i}}P_3^{m_3}P_4^{m_4}P_5^{m_5},\] where $m_{2,i}\ne m_{2,j}$ for all $i\ne j$, so that $G=\sum_{\vec{m}\in\mathbb{N}^3}G_{\vec{m}}$. A specialization $G(p_2,P_3,P_4,P_5)\in\overline{K}[P_3,P_4,P_5]$ is nonconstant in $P_3$ so long as $G_{\vec{m}}(p_2,P_3,P_4,P_5)\ne 0$ for some $\vec{m}=(m_3,m_4,m_5)$ such that $m_3>0$. Analogous statements hold for each of the variables $P_4,P_5$. Four successive claims result: \begin{enumerate}\item There is a finite set $\mathcal{Y}_2\subseteq\overline{K}$ such that if $p_j\notin\mathcal{Y}_2$, then $G(p_j,P_3,P_4,P_5)$ is nonconstant in each of $P_3,P_4,P_5$. If $D_2$ is the degree of $G$ in $P_2$, and $r$ is as in (\ref{eqn:leastmonomials}), then $|\mathcal{Y}_2|\le D_2r=:R=R(\mathcal{Z})$. \item Suppose we fix some $p_2\notin\mathcal{Y}_2$. Then there is a finite set $\mathcal{Y}_3=\mathcal{Y}_3(p_2)\subseteq\overline{K}$ of cardinality at most $R$ such that if $p_j\notin\mathcal{Y}_3$, then $G(p_2,p_j,P_4,P_5)$ is nonconstant in $P_4$ and $P_5$.\item Suppose we fix $p_2\notin\mathcal{Y}_2$, and subsequently fix some $p_3\notin\mathcal{Y}_3=\mathcal{Y}_3(p_2)$. Then there is a finite set $\mathcal{Y}_4=\mathcal{Y}_4(p_2,p_3)\subseteq\overline{K}$ of cardinality at most $R$ such that if $p_j\notin\mathcal{Y}_4$, then $G(p_2,p_3,p_j,P_5)$ is nonconstant (i.e., does not lie in $K$). \item Suppose we fix $p_2\notin\mathcal{Y}_2$, and subsequently fix $p_3\notin\mathcal{Y}_3=\mathcal{Y}_3(p_2)$, and finally fix some $p_4\notin\mathcal{Y}_4(p_2,p_3)$. Then there is a finite set $\mathcal{Y}_5=\mathcal{Y}_5(p_2,p_3,p_4)\subseteq\overline{K}$ of cardinality at most $R$ such that if $p_j\notin\mathcal{Y}_5$, then $G(p_2,p_3,p_4,p_j)\ne0$.\end{enumerate} 

On the other hand, Corollary \ref{cor:adgoodbars} states that if $\mathcal{T}_1=\{p_j-p_1:p_j\in K\text{ is preperiodic under }f\}$, with $|\mathcal{T}_1|=n$ and $n\gg_\epsilon1$, then $(1-\epsilon)n$ elements of $\mathcal{T}_1$ are adelically good. Let $\Sigma$ be the set of bad places of $K$ not dividing $2$ where the set of preperiodic points of $f$ in $K$ is $\epsilon$-equidistributed; assume first that $\Sigma$ is nonempty, and let $1>\delta>0$ be such that $\Sigma$ is a $\delta$-slice of places $v\in M_K^0\setminus\{v\mid2\}$. For $p_j-p_1\in\mathcal{T}_1$, write $a_j=p_j-p_1$. 

\textbf{Claim}: When $n\gg_\epsilon1$, there is an $\epsilon_2=\epsilon_2(\epsilon)>0$ such that $\epsilon_2n$ choices of $j$ satisfy \begin{equation}\label{ineq:longsides}\begin{split}\sum_{v\in\Sigma}r_v\lambda_v(a_j)&\ge\fr{1}{2}\left(\fr{1}{2}-\epsilon\right)\sum_{v\in\Sigma}r_v\lambda_v(c)\\&=\fr{1}{4}\left(1-2\epsilon\right)\sum_{v\in\Sigma}r_v\lambda_v(c).\end{split}\end{equation} 

Proof of claim: Suppose $\epsilon_3>0$ is such that there are strictly fewer than $\epsilon_3n$ elements of $\mathcal{T}_1$ satisfying (\ref{ineq:longsides}). By the definition of $\epsilon$-equidistribution, we have for $v\in\Sigma$ that \[\fr{1}{n}\sum_jr_v\lambda_v(a_j)\ge\left(\fr{1-\epsilon}{4}\right)r_v\lambda_v(c).\] Thus \begin{equation}\label{ineq:avglb}\sum_{v\in\Sigma}\left(\fr{1}{n}\sum_jr_v\lambda_v(a_j)\right)\ge\left(\fr{1-\epsilon}{4}\right)\sum_{v\in\Sigma}r_v\lambda_v(c).\end{equation} Then, since $\lambda_v(a_j)\le\fr{1}{2}\lambda_v(c)$, the definition of $\epsilon_3$ and (\ref{ineq:avglb}) give \begin{equation*}\begin{split}(1-\epsilon_3)\left(\fr{1}{4}-\fr{1}{2}\epsilon\right)\sum_{v\in\Sigma}r_v\lambda_v(c)+\epsilon_3\sum_{v\in\Sigma}\fr{1}{2}r_v\lambda_v(c)&>\fr{1}{n}\sum_{v\in\Sigma}\sum_jr_v\lambda_v(a_{j})\\&\ge\left(\fr{1-\epsilon}{4}\right)\sum_{v\in\Sigma}r_v\lambda_v(c).\end{split}\end{equation*} As $\Sigma$ is nonempty by assumption, this inequality is contradicted if $\epsilon_3$ is sufficiently small. This proves the claim. From Lemma \ref{lem:shape}, we know that for $v\nmid d$ a place of bad reduction, either $\lambda_v(a_j)=0$ or $\lambda_v(a_j)=\fr{1}{2}\lambda_v(c)$. Consequently, it follows from the claim that for each of these $\epsilon_2n$ choices of $j$, there is a $(\fr{1}{2}-\epsilon)$-slice of places $v\in\Sigma$ such that $|a_j|_v=|c|_v^{1/2}$. 

Moreover, we have seen in (i) that for all but finitely many $p_j\in K$, $G(p_j,P_3,P_4,P_5)$ is nonconstant in each of $P_3,P_4,P_5$. Therefore by Corollary \ref{cor:adgoodbars}, if $n\gg_{\mathcal{Z},\epsilon,\epsilon_2}1$, then there is a $p_2\in K$ preperiodic under $f$ such that $p_2-p_1$ is adelically good, $G(p_2,P_3,P_4,P_5)$ is nonconstant in each of $P_3,P_4,P_5$, and $|p_2-p_1|_v=|c|_v^{1/2}$ for a $(\fr{1}{2}-\epsilon)$-slice $\Sigma_2$ of places $v\in\Sigma$. Fix such a $p_2$ and $\Sigma_2$.

In a similar manner, $\epsilon_2n$ elements $a_j\in\mathcal{T}_1$ satisfy $|a_j|_v=|c|_v^{1/2}$ for a $(1/2-\epsilon)$-slice of places $v\in\Sigma_2$ (and hence a $(1/2-\epsilon)^2$-slice of places $v\in\Sigma$). Hence by (ii), if $n\gg_{\mathcal{Z},\epsilon,\epsilon_2}1$, we can fix a $p_3\in K$ preperiodic under $f$ so that $p_3-p_1$ is adelically good, $G(p_2,p_3,P_4,P_5)$ is nonconstant in $P_4$ and $P_5$, and $|p_3-p_1|_v=|c|_v^{1/2}$ for a $(1/2-\epsilon)$-slice $\Sigma_3$ of places $v\in\Sigma_2$.

Let \[\mathcal{T}_{-1}=\{p_j+p_1:p_j\in K\text{ is preperiodic under }f\}=-\mathcal{T}_1,\] and let \[\mathcal{T}_2=\{p_j-p_2:p_j\in K\text{ is preperiodic under }f\}.\] We have established that if $n\gg_\epsilon1$, there are $\epsilon_2n$ elements $p_j+p_1\in\mathcal{T}_{-1}$ such that $|p_j+p_1|_v=|c|_v^{1/2}$ for a $\fr{1-\epsilon}{2}$-slice $\Sigma_4$ of places $v\in\Sigma_3$. Additionally, there are $(1-\epsilon)n$ elements of $\mathcal{T}_{-1}$ that are adelically good, and $(1-\epsilon)n$ elements of $\mathcal{T}_2$ that are adelically good. Combining these two statements, it follows that if we assume $\epsilon<\fr{\epsilon_2}{4}$ (which we can do while leaving $\epsilon_2$ unchanged), then for $n\gg_{\epsilon}1$, there are at least $\fr{\epsilon_2}{2}n$ elements $p_j+p_1\in\mathcal{T}_{-1}$ such that both: \begin{itemize} \item $p_j-p_2$ and $p_j+p_1$ are adelically good \item $|p_j+p_1|_v=|c|_v^{1/2}$ for a $\fr{1-\epsilon}{2}$-slice $\Sigma_4$ of places $v\in\Sigma_3$. \end{itemize} Thus by (iii), if $n\gg_{\mathcal{Z},\epsilon,\epsilon_2}1$, we can fix a $p_j=p_4\in K$ having these two properties, such that $G(p_2,p_3,p_4,P_5)$ is nonconstant in $P_5$. Similarly by (iv), if $n\gg_{\mathcal{Z},\epsilon,\epsilon_2}1$, then there is a $p_j=p_5\in K$ preperiodic under $f$ such that $p_5-p_3$ and $p_5+p_1$ are both adelically good, $|p_5+p_1|_v=|c|_v^{1/2}$ for a $(1/2-\epsilon)$-slice of places $v\in\Sigma_4$, and $G(p_2,p_3,p_4,p_5)\ne0$. 

Let $w=(x_0,x_1,x_2,x_3,x_4,x_5)$ be the hexagon defined by $p_2,p_3,p_4,p_5$. By the construction of $w$, we have $w\in(K^*)^6$, and \[P=(x_0,x_1,x_2,x_3,x_4,x_5)\in\mathcal{H}\setminus\mathcal{Z}.\] Note also that if $v\in M_K^0\setminus\{v\mid2\}$ is such that $|p_2-p_1|_v=|p_3-p_1|_v=|p_4+p_1|_v=|p_5+p_1|_v$, then since $|p_1-(-p_1)|_v=|2p_1|_v=|c|_v^{1/2}$, we have \[|p_5-p_3|_v=|p_4-p_2|_v=|c|_v^{1/2}.\] Therefore \begin{equation}\label{eqn:allsidesequal}|x_0|_v=|x_1|_v=|x_2|_v=|x_3|_v=|x_4|_v=|x_5|_v=|c|_v^{1/2}\end{equation} for a $(\fr{1}{2}-\epsilon)^4$-slice of places $v\in\Sigma$, and $x_0,x_1,x_2,x_3,x_4,x_5$ are adelically good.

We will show that this is a contradiction for all $\delta$ sufficiently close to $1$ and all $c\in K$ with $h(c)\gg_{\xi,\epsilon,\delta,K}1$. (Recall that $1>\delta>0$ is defined to be a real number such that $\Sigma$ is a $\delta$-slice of places $v\in M_K^0\setminus\{v\mid2\}$.) Let $S_1$ be the set of places $v\in M_K^0\setminus\{v\mid2\}$ of bad reduction for $f$. Let $S_{1,1}$ be the set of bad places in $M_K^0\setminus\{v\mid2\}$ such that \[|x_0|_v=|x_1|_v=\cdots=|x_5|_v,\] and let $S_{1,2}=S_1\setminus S_{1,1}$. Let $S_2$ be the set of places $v\in M_K^0\setminus\{v\mid 2\}$ of good reduction for $f$. Let \[\text{rad}_v(P)=
\begin{cases}
N_\mathfrak{p} & \text{if }v=\mathfrak{p}\in\text{rad}(P)\\
0 & \text{otherwise}.
\end{cases}\] At all $v\nmid 2$, we have $|p_1-(-p_1)|_v=|c|_v^{1/2}$. Then by (\ref{eqn:allsidesequal}), \[\sum_{v\in S_{1,1}}r_v\log\max\{|x_0|_v,\dots,|x_5|_v\}-\textup{rad}_v(P)\ge\left(\fr{1}{2}-\epsilon\right)^4\sum_{v\in\Sigma}\fr{1}{2}r_v\lambda_v(c).\] Since we trivially have \[\sum_{v\in S_{1,2}}r_v\log\max\{|x_0|_v,\dots,|x_5|_v\}-\textup{rad}_v(P)\ge0,\] this yields \begin{equation*}\begin{split}\sum_{v\in S_1}r_v\log\max\{|x_0|_v,\dots,|x_5|_v\}-\textup{rad}_v(P)&\ge\left(\fr{1}{2}-\epsilon\right)^4\sum_{v\in\Sigma}\fr{1}{2}r_v\lambda_v(c)\\&\ge\fr{1}{2}\left(\fr{1}{2}-\epsilon\right)^4\delta\sum_{v\in M_K^0\setminus\{v\mid 2\}}r_v\lambda_v(c)\\&\ge\fr{1}{2}(1-\xi)\left(\fr{1}{2}-\epsilon\right)^4\delta h(c).\end{split}\end{equation*} On the other hand, \begin{equation}{\label{ineq:.04}}\sum_{v\in S_2}r_v\log\max\{|x_0|_v,\dots,|x_5|_v\}-\textup{rad}_v(P)\ge-\fr{1}{50}h(c).\end{equation} Indeed, since $x_0,x_1,\dots,x_5$ are adelically good, we have \[\sum_{v\in S_2}r_v\log\max\{|x_0|_v,\dots,|x_5|_v\}\ge-6\fr{1}{600}h(c),\] and so a fortiori, \[\sum_{v\in S_2}\textup{rad}_v(P)\le\fr{1}{100}h(c),\] yielding (\ref{ineq:.04}). Thus \[\sum_{v\in M_K^0\setminus\{v\mid2\}} r_v\log\max\{|x_0|_v,\dots,|x_5|_v\}-\textup{rad}_v(P)\ge-\fr{1}{50}h(c)+\fr{1}{2}(1-\xi)(1/2-\epsilon)^4\delta h(c).\] By (\ref{eqn:archplaces}), and the fact that for a number field $K$, \[\sum_{v=\mathfrak{p}\mid 2} N_\mathfrak{p}=\sum_{v=\mathfrak{p}\mid2}\fr{\log(\#k_\mathfrak{p})}{[K:\QQ]}\le\log 2,\] one obtains \[\sum_{v\in M_K^\infty\cup\{v\mid2\}} r_v\log\max\{|x_0|_v,\dots,|x_5|_v\}-\text{rad}_v(P)\ge-\fr{1}{800}h(c)-\log2.\] We conclude that \[\sum_{v\in M_K}r_v\log\max\{|x_0|_v,\dots,|x_5|_v\}-\textup{rad}_v(P)\ge\left(-\fr{1}{50}+\fr{\delta}{2}(1-\xi)\left(\fr{1}{2}-\epsilon\right)^4-\fr{1}{800}\right)h(c)-\log2.\] For all sufficiently small choices of $\xi,\epsilon$ along with any $\delta$ sufficiently near $1$, the quantity \begin{equation}\label{eqn:quantity}\left(-\fr{1}{50}+\fr{1}{2}\delta(1-\xi)\left(\fr{1}{2}-\epsilon\right)^4-\fr{1}{800}\right)h(c)\end{equation} is a positive proportion of $h(c)$. When $h(c)\gg_{\delta,\xi,\epsilon,K}1$, this contradicts Conjecture \ref{conj:nconj}. By Corollary \ref{cor:bucketequid}, we can assume $\delta$ is arbitrarily close to $1$, and $\epsilon$ is arbitrarily close to $0$, provided that $n\gg_{\delta,\epsilon}1$ and $\xi=\xi(\delta,\epsilon)$ is sufficiently small. Therefore there is an absolute constant $\Xi$ such that whenever \begin{equation}\label{eqn:Xi}\sum_{v\in M_K^0\setminus\{v\mid 2\}}r_v\lambda_v(c)\ge(1-\Xi)h(c)\end{equation} and $h(c)\gg_K1$, there are at most $B_1=B_1(K)$ preperiodic points of $f$ in $K$. On the other hand, Proposition \ref{prop:xi} implies that there is a $B_2=B_2(K)$ such that when $f$ fails to satisfy (\ref{eqn:Xi}), there are at most $B_2$ preperiodic points of $f$ contained in $K$. The proof is thus complete when $h(c)\gg_K1$ and $\Sigma$ is nonempty. To handle the remaining values of $c\in K$ (still assuming $\Sigma$ is nonempty), we apply Proposition \ref{prop:xi}, recalling that $h(c)>0$ by assumption if $K$ is a function field.

Finally, if $\Sigma$ is empty, then Corollary \ref{cor:bucketequid} and Proposition \ref{prop:xi} together imply that either $f$ has at most $B_3$ preperiodic points in $K$, where $B_3$ is an absolute constant, or that $K$ is a function field and $h(c)=0$. The latter has been excluded by hypothesis. Taking $B=\max\{B_1,B_2,B_3\}$ completes the proof of Theorem \ref{thm:UBCunicrit} when $d=2$.
\end{proof}

\section{Proof of Theorem \ref{thm:UBCunicrit} when $d=3,4$}\label{section:cubicquartic}

In this section, we prove Theorem \ref{thm:UBCunicrit} when $d=3$ or $4$. The proof is in most respects parallel to that of the quadratic case.

We fix throughout a degree $3\le d\le 4$ and a primitive $d$-th root of unity $\zeta_d$. Let $K$ be a number field or a one-variable function field of characteristic zero, and let $f(z)=z^d+c\in K[z]$. Assume without loss of generality that $\zeta_d\in K$. Instead of using $6$-tuples, as was done in the quadratic case, we will use $4$-tuples, as illustrated in Figure \ref{figure:quadrilateral}.

\begin{definition*} A \emph{quadrilateral with basepoint $p_1$} is an element \[(p_2-p_1,\zeta_dp_1-p_2,p_3-\zeta_dp_1,p_1-p_3)\in K^4\] for preperiodic points $p_1,p_2,p_3\in K$.\end{definition*}

\begin{figure}[H]
	\begin{tikzpicture}[auto,scale=0.15,>=stealth]
	
	\draw[->, ultra thick, black](17.3,20) -- (0,0);
	\draw[->, ultra thick, black](0,40) -- (17.3,20);
	\draw[->, ultra thick, black](0,0) -- (-17.3,20);
	\draw[->, ultra thick, black](-17.3,20) -- (0,40);
	
	\draw (-10,9) node {$x_0$};
	\draw (20,18) node {$p_3$};
	\draw (0,-2) node {$p_1$};
	\draw (-19,18) node {$p_2$};
	\draw (10,9) node {$x_1$};
	\draw (-10,32) node {$x_2$};
	\draw (10,32) node {$x_3$};
	\draw (0,42) node {$\zeta_d p_1$};
	\node at (0,0) {\pgfuseplotmark{*}};
	\node at (17.3,20) {\pgfuseplotmark{*}};
	\node at (-17.3,20) {\pgfuseplotmark{*}};
	\node at (0,40) {\pgfuseplotmark{*}};
	\end{tikzpicture}
	\caption{A quadrilateral with basepoint $p_1$.}
	\label{figure:quadrilateral}
\end{figure}
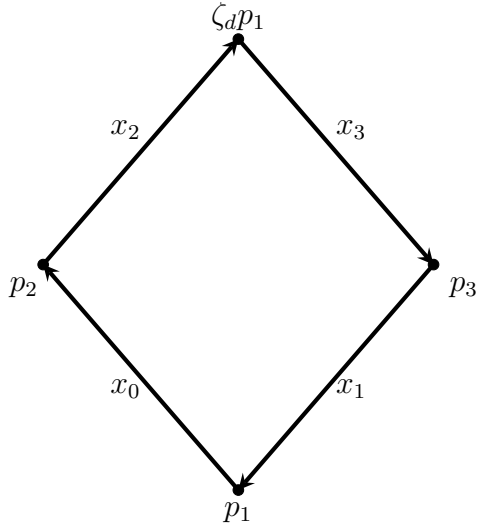

\begin{proof}[Proof of Theorem \ref{thm:UBCunicrit} when $d=3,4$] Let $\epsilon>0$, and let $\mathcal{Z}=\mathcal{Z}(K,\epsilon,4)$ be as in Conjecture \ref{conj:nconj}. Suppose $p_1\in K$ is a preperiodic point of $f$. By enlarging $\mathcal{Z}$ if necessary, we can assume $\mathcal{Z}$ is defined by a single homogeneous polynomial \[G(X_0,\dots,X_3)\in \overline{K}[X_0,\dots,X_3].\] For any quadrilateral $w=(x_0,x_1,x_2,x_3)$ with basepoint $p_1$, we have $x_0+x_1+x_2+x_3=0$. Thus, letting $P=(x_0,x_1,x_2,x_3)\in\mathbb{P}^3(K)$, we have $P\in\mathcal{H}$, where $\mathcal{H}$ is the hyperplane in $\mathbb{P}^3(K)$ given by $X_0+\dots+X_3=0$. We would like to produce a quadrilateral $w=(x_0,x_1,x_2,x_3)$ such that the associated point $P\in\mathbb{P}^3(K)$ satisfies $P\in\mathcal{H}\setminus\mathcal{Z}$. Making the substitutions $X_0=P_2-p_1,X_1=P_3-p_1,X_2=\zeta_dp_1-P_2,X_3=p_1-P_3$, we can write \[G(X_0,\dots,X_3)=G(P_2,P_3)=\sum_{l=1}^ra_lP_2^{i_{l,2}}P_3^{i_{l,3}}\in\overline{K}[P_2,P_3].\] For $m_3\in\mathbb{N}$, write \[G_{m_3}=\sum_ib_iP_2^{m_{2,i}}P_3^{m_3},\] and $m_{2,i}\ne m_{2,j}$ for all $i\ne j$, so that $G=\sum_{m_3\in\mathbb{N}}G_{m_3}$. A specialization $G(p_2,P_3)$ is nonconstant so long as $G_{m_3}\ne 0$ for some $m_3>0$. Thus: \begin{enumerate} \item There is a finite set $\mathcal{Y}_2\subseteq\overline{K}$ such that $G(p_2,P_3)$ is nonconstant for any $p_2\notin\mathcal{Y}_2$. The cardinality of $\mathcal{Y}_2$ is at most some $R$ depending only on $\mathcal{Z}$. \item Suppose we fix $p_2\notin\mathcal{Y}_2$. Then there is a set $\mathcal{Y}_3=\mathcal{Y}_3(p_2)\subseteq\overline{K}$, with $|\mathcal{Y}_3|\le R$, such that for any $p_3\notin\mathcal{Y}_3$, $G(p_2,p_3)\ne0$.\end{enumerate}

Let $T$ be the set of $K$-rational preperiodic points of $f$, with $|T|=n$. For $v\in M_K^0\setminus\{v\mid d\}$ a place of bad reduction for $f$, and $p_j\in K$ preperiodic under $f$, let \begin{equation*}\chi_v(p_j)=\begin{cases}r_v\lambda_v(c) & \text{ if }|p_j-p_1|_v=|p_j-\zeta_d p_1|_v=|c|_v^{1/d}\\ 0 & \text{ otherwise.}\end{cases}\end{equation*} Let $\Sigma$ be the set of places $v\in M_K^0\setminus\{v\mid d\}$ of bad reduction for $f$ such that $T$ is $\epsilon$-equidistributed at $v$. Assume $\Sigma$ is nonempty, and let $1>\delta>0$ be such that $\Sigma$ is a $\delta$-slice of places $v\in M_K^0\setminus\{v\mid d\}$. Let $\xi=\xi(\epsilon)$, $M=M(\epsilon)$, and $N=N(\epsilon)$ be as in Corollary \ref{cor:adgoodbars}. Suppose $n\ge N$, $h(c)\ge M$, and \[\sum_{v\in M_K^0\setminus\{v\mid d\}}r_v\lambda_v(c)\ge(1-\xi)h(c).\]

\textbf{Claim 1}: There is an $\epsilon_2=\epsilon_2(\epsilon)>0$ such that $\epsilon_2n$ elements $p_j\in T$ satisfy \begin{equation}\label{ineq:longlong}\sum_{v\in\Sigma}\chi_v(p_j)\ge\fr{d-2}{d}(1-2\epsilon)\sum_{v\in\Sigma}r_v\lambda_v(c).\end{equation} 

Proof of Claim 1: At any $v\in\Sigma$, we have by Lemma \ref{lem:shape} and the definition of $\epsilon$-equidistribution that \[\fr{1}{n}\sum_{p_j\in T}\chi_v(p_j)\ge\fr{d-2}{d}(1-\epsilon)r_v\lambda_v(c).\] Thus \begin{equation}\label{eqn:avglb}\fr{1}{n}\sum_{v\in\Sigma}\sum_{p_j\in T}\chi_v(p_j)\ge\fr{d-2}{d}(1-\epsilon)\sum_{v\in\Sigma}r_v\lambda_v(c).\end{equation} Suppose $\epsilon_3>0$ is such that there are strictly fewer than $\epsilon_3n$ elements of $T$ satisfying (\ref{ineq:longlong}). Then \[\fr{1}{n}\sum_{v\in\Sigma}\sum_j\chi_v(p_j)<\epsilon_3\cdot \sum_{v\in\Sigma}r_v\lambda_v(c)+(1-\epsilon_3)\sum_{v\in\Sigma}\frac{d-2}{d}(1-2\epsilon)r_v\lambda_v(c).\] As $\epsilon_3\to0$, the right-hand side approaches \begin{equation*}\sum_{v\in\Sigma}\dfrac{d-2}{d}(1-2\epsilon)r_v\lambda_v(c)<\sum_{v\in\Sigma}\dfrac{d-2}{d}(1-\epsilon)r_v\lambda_v(c),\end{equation*} contradicting (\ref{eqn:avglb}) since we have assumed $\Sigma$ is nonempty. This proves Claim 1. Since, by Lemma \ref{lem:shape}, $\chi_v(p_j)$ is either $0$ or $r_v\lambda_v(c)$, it follows that for the $\epsilon_2 n$ elements of $T$ satisfying (\ref{ineq:longlong}), there is a $\fr{d-2}{d}(1-2\epsilon)$-slice of places $v\in\Sigma$ such that \[|p_j-p_1|_v=|p_j-\zeta_dp_1|_v=|c|_v^{1/d}.\]

By Corollary \ref{cor:adgoodbars}, for $n\gg_\epsilon1$, there are at least $(1-\epsilon)n$ elements $p_j\in T$ such that $p_j-p_1$ and $p_j-\zeta_dp_1$ are adelically good. Adjusting $\epsilon$ if needed so that $\epsilon<\fr{\epsilon_2}{4}$ (which we can do while leaving $\epsilon_2$ unchanged), it follows that at least $\fr{\epsilon_2n}{2}$ elements $p_j\in T$ are such that: \begin{itemize} \item $p_j-p_1$ and $p_j-\zeta_dp_1$ are adelically good \item $|p_j-p_1|_v=|p_j-\zeta_dp_1|_v=|c|_v^{1/d}$ for a $\fr{d-2}{d}(1-2\epsilon)$-slice of places $v\in\Sigma$ \end{itemize} Thus if $n\gg_{\mathcal{Z},\epsilon,\epsilon_2}1$, there is a $p_j\in T$ having both of these properties, such that $G(p_j,P_3)$ is nonconstant in $P_3$. Fix such a $p_j=:p_2$. Let $\Sigma_{p_2}$ be a $\fr{d-2}{d}(1-2\epsilon)$-slice of places $v\in\Sigma$ such that $|p_2-p_1|_v=|p_2-\zeta_dp_1|_v=|c|_v^{1/d}$. 

\textbf{Claim 2}: There is an $\epsilon_2=\epsilon_2(\epsilon)>0$ such that $\epsilon_2 n$ elements $p_j\in T$ satisfy \[\sum_{v\in\Sigma_{p_2}}\chi_v(p_j)\ge\fr{d-2}{d}(1-2\epsilon)\sum_{v\in\Sigma_{p_2}}r_v\lambda_v(c)\]

Proof of Claim 2: The proof is similar to that of Claim 1, with $\Sigma_{p_2}$ replacing $\Sigma$. 

Therefore, assuming $n\gg_{\mathcal{Z},\epsilon,\epsilon_2}1$, there is a $p_j\in T$ with the property that $p_j$ satisfies \[|p_j-p_1|_v=|p_j-\zeta_dp_1|_v=|c|_v^{1/d}\] for a $\fr{d-2}{d}(1-2\epsilon)$-slice of places $v\in\Sigma_{p_2}$, where $p_j-p_1$ and $p_j-\zeta_dp_1$ are adelically good, and $G(p_2,p_j)\ne 0$. Fix such a $p_j=:p_3$, and let \[w=(x_0,x_1,x_2,x_3)=(p_2-p_1,p_1-p_3,\zeta_dp_1-p_2,p_3-\zeta_dp_1).\] Then \[|x_0|_v=|x_1|_v=|x_2|_v=|x_3|_v\] for a $(\fr{d-2}{d}(1-2\epsilon))^2$-slice of places $v\in\Sigma$, and $x_0,x_1,x_2,x_3$ are adelically good. By the construction of $w$, one sees that $P=(x_0,x_1,x_2,x_3)\in\mathcal{H}\setminus\mathcal{Z}$.

We will show that this is a contradiction for all $\delta$ sufficiently close to $1$ and all $c\in K$ with $h(c)\gg_{\xi,\epsilon,\delta,K}1$. Let $S_1$ be the set of places $v\in M_K^0\setminus\{v\mid d\}$ of bad reduction for $f$. Let $S_{1,1}$ be the set of bad places in $M_K^0\setminus\{v\mid d\}$ such that $|x_0|_v=|x_1|_v=|x_2|_v=|x_3|_v$, and let $S_{1,2}=S_1\setminus S_{1,1}$. Let $S_2$ be the set of places $v\in M_K^0\setminus\{v\mid d\}$ of good reduction for $f$. Let \[\text{rad}_v(P)=
\begin{cases}
N_\mathfrak{p} & \text{if }v=\mathfrak{p}\in\text{rad}(P)\\
0 & \text{otherwise}.
\end{cases}\]

As $d=3$ or $d=4$, we have $(\fr{d-2}{d})^2\ge\fr{1}{9}$. At all $v\nmid d$, we have $|p_1-(\zeta_dp_1)|_v=|c|_v^{1/d}$. Thus \begin{equation*}\begin{split}\sum_{v\in S_{1,1}}r_v\log\max\{|x_0|_v,|x_1|_v,|x_2|_v,|x_3|_v\}-\textup{rad}_v(P)&\ge\left(\fr{d-2}{d}(1-2\epsilon)\right)^2\sum_{v\in\Sigma}\fr{1}{d}r_v\lambda_v(c)\\&\ge\fr{1}{9}(1-2\epsilon)^2\sum_{v\in\Sigma}\fr{1}{d}r_v\lambda_v(c).\end{split}\end{equation*} Since it holds trivially that \[\sum_{v\in S_{1,2}}r_v\log\max\{|x_0|_v,|x_1|_v,|x_2|_v,|x_3|_v\}-\textup{rad}_v(P)\ge0,\] this yields \begin{equation*}\begin{split}\sum_{v\in S_1}r_v\log\max\{|x_0|_v,|x_1|_v,|x_2|_v,|x_3|_v\}-\textup{rad}_v(P)&\ge\fr{1}{9}(1-2\epsilon)^2\sum_{v\in \Sigma}\fr{1}{d}r_v\lambda_v(c)\\&\ge\fr{1}{9}(1-2\epsilon)^2\delta\sum_{v\in M_K^0\setminus\{v\mid d\}}\fr{1}{d}r_v\lambda_v(c)\\&\ge\fr{\delta}{9d}(1-\xi)(1-2\epsilon)^2h(c).\end{split}\end{equation*} On the other hand, \begin{equation}{\label{ineq:8/600}}\sum_{v\in S_2}r_v\log\max\{|x_0|_v,|x_1|_v,|x_2|_v,|x_3|_v\}-\textup{rad}_v(P)\ge-\fr{8}{600}h(c).\end{equation} Indeed, since $x_0,x_1,x_2,x_3$ are adelically good, we have \[\sum_{v\in S_2}r_v\log\max\{|x_0|_v,|x_1|_v,|x_2|_v,|x_3|_v\}\ge-4\fr{1}{600}h(c),\] and so a fortiori, \[\sum_{v\in S_2}\textup{rad}_v(P)\le\fr{4}{600}h(c),\] yielding (\ref{ineq:8/600}). Hence \[\sum_{v\in M_K^0\setminus\{v\mid d\}} r_v\log\max\{|x_0|_v,|x_1|_v,|x_2|_v,|x_3|_v\}-\textup{rad}_v(P)\ge-\fr{8}{600}h(c)+\fr{\delta}{9d}(1-\xi)(1-2\epsilon)^2h(c).\] By (\ref{eqn:archplaces}), and the fact that for a number field $K$, \[\sum_{v=\mathfrak{p}\mid d} N_\mathfrak{p}=\sum_{v=\mathfrak{p}\mid d}\fr{\log(\#k_\mathfrak{p})}{[K:\QQ]}\le\log d\] one obtains \[\sum_{v\in M_K^\infty\cup\{v\mid d\}} r_v\log\max\{|x_0|_v,|x_1|_v,|x_2|_v,|x_3|_v\}-\text{rad}_v(P)\ge-\fr{1}{800}h(c)-\log d.\] We conclude that \begin{equation*}\sum_{v\in M_K}r_v\log\max\{|x_0|_v,\dots,|x_3|_v\}-\textup{rad}_v(P)\ge\left(-\fr{8}{600}+\fr{\delta}{9d}(1-\xi)(1-2\epsilon)^2-\fr{1}{800}\right)h(c)-\log d.\end{equation*} For all sufficiently small choices of $\xi$ and $\epsilon$, and for any $\delta$ sufficiently close to $1$, the quantity \[\left(-\fr{8}{600}+\fr{\delta}{9d}(1-\xi)(1-2\epsilon)^2-\fr{1}{800}\right)h(c)\] is a positive proportion of $h(c)$. Fixing such choices of $\xi,\epsilon$, and $\delta$, this contradicts Conjecture \ref{conj:nconj} when $h(c)\gg_K1$. The rest of the proof proceeds exactly as in the quadratic case, from the line after (\ref{eqn:quantity}) to the end of \S\ref{section:d2}.\end{proof}

\section{Proof of Theorem \ref{thm:UBCunicrit} when $d\ge5$}{\label{section:d5}}

As noted in the introduction, the ideas used in proving Theorem \ref{thm:UBCunicrit} when $2\le d\le 4$ may be adapted to the case $d\ge5$, using triangles instead of hexagons or quadrilaterals. One thus uses the case $n=3$ of Conjecture \ref{conj:nconj}, which is the usual $abc$-conjecture. (One difference is that the triangles in question have no basepoint.) 
To avoid repetition and introduce alternative approaches, however, we choose to use a separate technique, one which works specifically when $d\ge5$. This method also uses the $abc$-conjecture. In contrast to our previous method, however, it relies purely on algebraic arguments.

Fix $K$ a number field or a one-variable function field of characteristic zero, and let $f(z)=z^d+c\in K[z]$. For $\alpha\in K^*$, and $S$ the set of places $\mathfrak{p}\in M_K^0$ such that $v_\mathfrak{p}(\alpha)>0$, let \[\textup{rad}(\alpha)=\sum_{\mathfrak{p}\in S}N_\mathfrak{p}.\] Lemmas \ref{lem:localhtbound}, \ref{lem:htofdiff}, and \ref{lem:coprimality} will be used to prove Proposition \ref{prop:lowerhtbound}, the key ingredient in the proof of Theorem \ref{thm:UBCunicrit} when $d\ge5$. \begin{lem}\cite[Lemmas 3 and 6]{Ingram:mincanht}{\label{lem:localhtbound}} Let $v\in M_K$, and let $\alpha\in\CC_v$. If $\hat{\lambda}_v(\alpha)=0$, then $\alpha\in D(0,2_v|c|_v^{1/d})$.\end{lem} \begin{lem}{\label{lem:htofdiff}} Let $d\ge2$, and let $p_1,p_2$ be preperiodic points of $f(z)=z^d+c\in K[z]$. Then \[h(p_1-p_2)\le \fr{1}{d}h(c)+\log4.\]
\end{lem} \begin{proof} From Lemma \ref{lem:localhtbound}, we have $\lambda_v(p_1),\lambda_v(p_2)\le\fr{1}{d}\lambda_v(c)+\log2$ for all $v\in M_K^\infty$, and $\lambda_v(p_1),\lambda_v(p_2)\le\fr{1}{d}\lambda_v(c)$ for all $v\in M_K^0$. If $v\in M_K^\infty$, then \begin{equation}{\label{eqn:arch}}\lambda_v(p_1-p_2)\le\max\{\lambda_v(p_1),\lambda_v(p_2)\}+\log2\le\fr{1}{d}\lambda_v(c)+2\log2.\end{equation} If $v\in M_K^0$, then \begin{equation}{\label{eqn:nonarch}}\lambda_v(p_1-p_2)\le\max\{\lambda_v(p_1),\lambda_v(p_2)\}\le\fr{1}{d}\lambda_v(c).\end{equation} Summing (\ref{eqn:arch}) and (\ref{eqn:nonarch}) over all places completes the proof.
\end{proof} The \emph{period} of a periodic point $\alpha\in\overline{K}$ of $f$ is the smallest $n$ such that $f^n(\alpha)=\alpha$. A \emph{cycle} of length $n$ is the forward orbit of a periodic point of period $n$. \begin{lem}{\label{lem:coprimality}} Let $d\ge2$, and let $f(z)=z^d+c\in K[z]$. Let $p_1,p_2\in K$ be distinct periodic points of $f$ having the same period. If $\mathfrak{p}\in M_K^0$ and $v_\mathfrak{p}(p_1)>0$, then $v_\mathfrak{p}(p_2)=0$. 
\end{lem} \begin{proof} Any prime $\mathfrak{p}\in M_K^0$ such that $v_\mathfrak{p}(p_1)>0$ is a prime of good reduction for $f$. From the Newton polygon of $f^n(z)-z$, where $n$ is the period of $p_1,p_2$, we see that any prime $v$ such that $v_\mathfrak{p}(p_1)>0$ must satisfy $v_\mathfrak{p}(f^n(0))>0$, and $v_\mathfrak{p}(p_i)=0$ for any other root $p_i\ne p_1$ of $f^n(z)-z$.
\end{proof} \begin{prop}{\label{prop:lowerhtbound}} Fix $d\ge2$, and let $f(z)=z^d+c\in K[z]$. Let $p_1,p_2\in K$ be distinct periodic points of $f$ having the same period. Let $P=(p_1^d,p_2^d,f(p_1)-f(p_2))\in\mathbb{P}^2(K)$. Suppose $\xi$ is such that \[\sum_{v\in M_K^\infty}r_v\lambda_v(c)\le\xi h(c).\] Then there exists a constant $\eta$ such that \[h(P)\ge\fr{d-1-\xi}{d}h(c)+\textup{rad}(P)-\eta.\]
\end{prop} \begin{proof} For each $v\in M_K^\infty$, \cite[Lemma 4]{Ingram:mincanht} implies that we have \begin{equation}{\label{ineq:archPilb}}\lambda_v(p_i)\ge\fr{1}{d}\lambda_v(c)-C\end{equation} for some constant $C$ depending only on $d$. For $P=(z_1,z_2,z_3)\in\mathbb{P}^2(K)$ with $z_1,z_2,z_3\in K$, let \begin{equation*}\begin{split}h_+(P)=&\sum_{\textup{primes }\mathfrak{p} \textup{ of } K} -\min\{0,v_{\mathfrak{p}}(z_1),v_{\mathfrak{p}}(z_2),v_{\mathfrak{p}}(z_3)\}N_{\mathfrak{p}}\\&+\dfrac{1}{[K:\QQ]} \sum_{\sigma:K\hookrightarrow\CC} \max\{0,\textup{log}|\sigma(z_1)|,\textup{log}|\sigma(z_2)|,\textup{log}|\sigma(z_3)|\}\end{split}\end{equation*} if $K$ is a number field, and let \begin{equation*}h_+(P)=\sum_{\textup{primes }\mathfrak{p} \textup{ of } K} -\min\{0,v_{\mathfrak{p}}(z_1),v_{\mathfrak{p}}(z_2),v_{\mathfrak{p}}(z_3)\}N_{\mathfrak{p}}\end{equation*} if $K$ is a function field. Let \begin{equation}{\label{eqn:htdecomp}}h_-(P)=h_+(P)-h(P)\ge0,\end{equation} and let $P=(p_1^d,p_2^d,f(p_1)-f(p_2))\in\mathbb{P}^2(K)$. We wish to bound $h(P)$ from below. By Lemma \ref{lem:coprimality} and (\ref{ineq:archPilb}), we have \begin{equation}{\label{eqn:h+bd}}h_+(P)\ge h(c)-\eta'\end{equation} for some $\eta'$ depending only on $d$. To bound $h_-(P)$ from above, it suffices to bound \[\sum_{v\in M_K^\infty}-r_v\log\min\{1,|f(p_1)-f(p_2)|_v\}\] from above, since by Lemma \ref{lem:coprimality}, \begin{equation*}\begin{split}h_-(P)&=\sum_{v\in M_K^\infty}-r_v\log\min\{1,\max\{|p_1^d|_v,|p_2^d|_v,|f(p_1)-f(p_2)|_v\}\}\\&\le\sum_{v\in M_K^\infty}-r_v\log\min\{1,|f(p_1)-f(p_2)|_v\}.\end{split}\end{equation*} But (noting that $p_1\ne p_2$ implies $f(p_1)\ne f(p_2)$), letting \[S=\{\mathfrak{p}\in M_K^0:v_\mathfrak{p}(f(p_1)-f(p_2))>0\}\] gives \begin{equation*}\begin{split}\sum_{\mathfrak{p}\in M_K^0} v_\mathfrak{p}(f(p_1)-f(p_2))N_\mathfrak{p}&\ge\sum_{\mathfrak{p}\in S}N_\mathfrak{p}+\sum_{v\in M_K^0\setminus S}-\fr{1}{d}r_v\lambda_v(c)\\&\ge\textup{rad}(f(p_1)-f(p_2))-\fr{1}{d}h(c).\end{split}\end{equation*} By the product formula, we thus have \begin{equation*}\begin{split}\sum_{v\in M_K^\infty}-r_v\log\min\{1,|f(p_1)-f(p_2)|_v\}=&\sum_{v\in M_K^\infty}r_v\log^+|f(p_1)-f(p_2)|_v-\sum_{\mathfrak{p}\in M_K^0}v_\mathfrak{p}(f(p_1)-f(p_2))N_\mathfrak{p}\\\le&\sum_{v\in M_K^\infty}r_v\left(\fr{1}{d}\lambda_v(c)+2\log2\right)-\sum_{\mathfrak{p}\in M_K^0}v_\mathfrak{p}(f(p_1)-f(p_2))N_\mathfrak{p}\\\le&\fr{1}{d}\xi h(c)+2\log2+\fr{1}{d}h(c)-\textup{rad}(f(p_1)-f(p_2)).\end{split}\end{equation*} It follows from combining this with (\ref{eqn:htdecomp}) and (\ref{eqn:h+bd}) that \[h(P)\ge\fr{d-1-\xi}{d}h(c)+\textup{rad}(f(p_1)-f(p_2))-\eta\] for some constant $\eta$ depending only on $d$. \end{proof} We are now ready to finish the proof of Theorem \ref{thm:UBCunicrit}.\begin{proof}[Proof of Theorem \ref{thm:UBCunicrit} when $d\ge5$] Fix $d\ge5$. If $K$ is a function field, then the theorem follows from \cite[Theorem 1.6]{DoylePoonen}, so assume $K$ is a number field. Let $1>\xi>0$, and suppose $f(z)=z^d+c\in K[z]$ is such that \[\sum_{v\in M_K^\infty\cup\{v\mid d\}}r_v\lambda_v(c)\le\xi h(c).\] By \cite[Corollary 1.8]{DoylePoonen}, it suffices to prove that there is a uniform bound on the length of a $K$-rational cycle of $f$. Suppose $p_1,p_2\in K$ are distinct elements in a cycle of length $n\ge2$. Let $P=(p_1^d,p_2^d,f(p_1)-f(p_2))\in\mathbb{P}^2(K)$. For any preperiodic point $p_i\in K$, one has $v_\mathfrak{p}(p_i)=\fr{1}{d}v_\mathfrak{p}(c)$ for primes $\mathfrak{p}\in M_K^0$ such that $v_\mathfrak{p}(c)<0$. Thus \begin{equation}{\label{ineq:radbound}}\begin{split}\textup{rad}(P)&\le\textup{rad}(p_1)+\textup{rad}(p_2)+\textup{rad}(f(p_1)-f(p_2))+\fr{1}{d}h(c) \\&\le h(p_1)+h(p_2)+\textup{rad}(f(p_1)-f(p_2))+\fr{1}{d}h(c) \\&\le \fr{3}{d}h(c)+2\log2+\textup{rad}(f(p_1)-f(p_2)).\end{split}\end{equation} Therefore, for $\epsilon>0$, Proposition \ref{prop:lowerhtbound} combined with (\ref{ineq:radbound}) yields \begin{equation*}
	\begin{split}h(P)-(1+\epsilon)\textup{rad}(P)\ge&\fr{d-1-\xi}{d}h(c)+\textup{rad}(f(p_1)-f(p_2))-\eta\\&-(1+\epsilon)\left(\fr{3}{d}h(c)+\log4+\textup{rad}(f(p_1)-f(p_2))\right)\\\ge&\fr{d-4-\xi-3\epsilon}{d}h(c)-\epsilon\textup{rad}(f(p_1)-f(p_2))-(1+\epsilon)\log4-\eta\\\ge&\fr{d-4-\xi-4\epsilon}{d}h(c)-(1+2\epsilon)\log4-\eta,\end{split}
	\end{equation*} where the final inequality follows from Lemma \ref{lem:htofdiff}. For any sufficiently small choices of $\epsilon$ and $\xi$, we see that for all $c\in K$ with $h(c)\gg_{K,\epsilon,\xi}1$, this is greater than $C_{K,\epsilon}$, contradicting the $abc$-conjecture. By Northcott's Theorem, it follows that there is an upper bound on the number of $K$-rational \emph{periodic} points of $f(z)=z^d+c$ that is uniform across all $c\in K$. Applying \cite[Corollary 1.8]{DoylePoonen} then completes the proof. Proposition \ref{prop:xi} addresses the case where $\sum_{v\in M_K^\infty\cup\{v\mid d\}}r_v\lambda_v(c)>\xi h(c)$.\end{proof}

\end{document}